\documentclass[10pt]{amsart}

\usepackage[a4paper]{geometry}

\usepackage{graphicx}
\usepackage{float}
\usepackage{amsmath,amsthm, amsfonts, amssymb}

\newtheorem{theorem}{Theorem}[section]
\newtheorem{lemma}{Lemma}[section]
\newtheorem{proposition}{Proposition}[section]

\theoremstyle{definition}
\newtheorem{definition}{Definition}[section]
\theoremstyle{remark}
\newtheorem{remark}{Remark}[section]

\newcommand{\real} {\mathbb{R}}
\newcommand{\enteros} {\mathbb{Z}}

\newcommand{\mapto}{\longrightarrow}

\newcommand{\itself}{\circlearrowleft}

\newcommand{\diam}{\mbox{diam}\, }
\newcommand{\dist}{\mbox{dist}\, }
\newcommand{\Vol}{\mbox{Vol}\, }

\title{Yet another induction scheme \\ for non-uniformly expanding transformations}
\author{Pedro L. Capett-Figueras \& \\ Fernando J. S\'anchez-Salas}

\address{Departamento de Matem\'aticas, Facultad Experimental de Ciencias, Universidad del Zulia, Avenida Universidad, Edificio Grano de Oro, 
Maracaibo, Venezuela}
\email{pcapett@demat-fecluz.org / fjss@fec.luz.edu.ve}

\date{February 25, 2017}



\begin{document}

\begin{abstract}
We introduce a new induction scheme for non-uniformly expanding maps $f$ of compact Riemannian manifolds, relying upon ideas 
of \cite{sanchez.salas.2003} and \cite{doob.2012}. We use this induction approach to prove that the existence of a 
Gibbs-Markov-Young structure is a necessary condition for $f$ to preserve an absolutely continuous probability with all its 
Lyapunov exponents positive.
\end{abstract}

\maketitle

\section{Introduction}
A Borel probability $\mu$ is $f$-invariant if $\mu(f^{-1}(B)) = \mu(B)$ for every Borel subset $B \subset M$. Invariant 
probabilities provide quantitative information about the spatial distribution of orbits of $f$ and are crucial in the study 
of statistical properties of the dynamics. See \cite{manhe} and \cite[Chapter 4]{katok.hasselblat}.

Absolutely continuous invariant measures (\textbf{acim}) are of interest because they are generated by sets of positive Lebesgue measure, 
that is, \emph{they are experimentally observable}.

This paper is concerned with \emph{induction}, a useful notion aimed at the study of ergodic properties of dynamical systems. Namely, 
suppose that every point in a non-empty subset $B \subset M$ returns infinitely many often to $B$. This is the case, for example, 
of a set $B$ with $\mu(B) > 0$, by Poincar\'e's recurrence theorem. See \cite{manhe}. Then it is natural to consider the 
\emph{induced dynamical system}  $f^R : B \to B$ defined as a \emph{first return map} $f^R(x) = f^{R(x)}(x)$, 
where $R(x) = \min\{n > 0: f^n(x) \in B\}$ is the \emph{first return times} from $x \in B$ to $B$. 

First return maps were originally introduced by Poinca\-r\'e to study nearby orbits of limit cycles. It was later, in the work of Kakutani \cite{kakutani.1943}, 
that they revealed their usefulness in the study of ergodic properties of the dynamics. Transformations induced by more 
general return times were also considered in the work of J. Neveu and Schweiger between 1969 and 1981. See references in \cite{pesin.senti.zhang.2016}.

To the best of our knowledge, Jakobson's theorem \cite{jakobson.1981} was the first to show the close connection between induced maps 
and the existence of \textbf{acim}. Jakobson's theorem says that given a one-parameter family of real quadratic maps $f_a(x) = 1 - ax^2$, 
there exists a Cantor set $\Lambda$ with positive Lebesgue measure such that for every $a \in \Lambda$, $f = f_a$ has an ergodic \textbf{acim} 
with positive Lyapunov exponent. See \cite[Theorem A, Theorem B]{jakobson.1981}.

The idea of the proof is to construct $\Lambda$ in such way that for every $a \in \Lambda$, there exists a closed interval $J$ and an induced 
expanding Markov transformation $f^R : \bigcup_iJ_i \to J$ with bounded distortion (see subsection \ref{subsec:enunciados}) 
and integrable return times $R$. By a folklore theorem $f^R$ has a (unique) ergodic $f^R$-invariant \textbf{acim}. Integrability of return times 
permits to ``coinduce'' a $f$-invariant \textbf{acim} $\mu$, by an standard procedure (see below). We refer to Yoccoz's College de France manuscript 
\cite{yoccoz} for an up-to-date and quite clear exposition of this celebrated result.

Jakobson's paper inspired a number of results concerning the application of induction method to prove the existence of 
\textbf{acim} for different classes of one-dimensional dynamical systems. See for instance  \cite[Chapter V, Section 3]{demelo}, 
\cite{guckenheimer.johnson}, \cite{jakobson.swiatek}, \cite{martens} and \cite{vargas}.

The next breakthrough in the induction approach was Lai Sang Young's \emph{``horseshoes with infinitely many branches and variable 
return times''}. See \cite{young.1998}, \cite{young.1999}. Generally speaking, Young's horseshoes are Cantor sets endowed with an induced 
return map with good hyperbolic and distortion properties with respect to a reference measure $m$. The return structures are made of infinitely many 
'vertical strips' mapped onto  'horizontal strips' similarly to Smale's horseshoes. Projecting along stable invariant manifolds one 
gets a uniformly expanding map with good distortion properties $f^R: \bigcup_i\Delta_i \to \Delta$ defined over a decomposition of a Cantor set $\Delta$ 
with positive reference measure. These Markov structures are used to prove the existence of equilibrium measures, showing the connection 
between the speed of convergence to the equilibrium and its rates of mixing with the rate of decaying of return times tails.

The following is a partial sample of the growing body of literature on the applications of also called \emph{Markov towers} or 
\emph{Gibbs-Markov-Young structures} (GMY) in the study of statistical properties of the dynamics.  See \cite{alves.dias.luzzatto}, 
\cite{alves.luzzatto.pinheiro}, \cite{benedicks.young}, \cite{bruin.1995}, \cite{bruin.luzzatto.vanstrien}, \cite{chernov.young.2001}, \cite{doob.2012}, 
\cite{gouezel}, \cite{melbourne.2005}, \cite{melbourne.2009}, \cite{pesin.senti.2005}, \cite{pesin.senti.2008}, \cite{pesin.senti.zhang.2008}, 
\cite{pesin.senti.zhang.2016}, \cite{przyticky.rivera-letelier.2007}, \cite{przyticky.rivera-letelier.2011}.

In \cite{young.1998} Young describes informally diffeomorphisms exhibiting GMY structures as those being \emph{``hyperbolic in large 
part of the phase space without being uniformly hyperbolic''}, however, up to the best of our knowledge, a precise characterization of the systems 
displaying Markov structures is still lacking.

Recently, it was conjectured that a \emph{$C^{1+\alpha}$ diffeomorphism of compact Riemannian 
manifold has a GMY structure if and only if it is non-uniformly hyperbolic}, meaning a system with an invariant measure with non zero 
Lyapunov exponents. The conjecture was proved in \cite{alves.dias.luzzatto} for a large class of non-uniformly expanding systems. In particular, 
let $f$ be a $C^{1+\alpha}$ map of compact Riemannian manifold of arbitrary dimension with critical points. If 
the criticalities satisfy non-degeneracy conditions similar to non-flatness in one-dimensional maps, then it is proved that $f$ preserves an 
ergodic \textbf{acim} $\mu$ with all its Lyapunov exponents positive if and only if it there exists a GMY structure 
having Lebesgue as a reference measure. If, in addition, $f$ has singularities, then a similar result was proved whenever $\mu$ satisfies a 
regularity condition. See \cite[Theorem 2, Theorem 3]{alves.dias.luzzatto} and subsection \ref{subsec:enunciados}. 

The paper \cite{alves.dias.luzzatto} proposes a simplified and conceptually clearer method to induce an expanding Markov map with good 
distortion properties, in comparison with the technicalities involved in most of the results on the subject. The proposed algorithm is essentially 
the most naive strategy to construct an induced expanding transformation in the non-uniformly expanding setting. The idea is to fix an open ball 
$B$ and iterate it until the return of some subset of $B'\subset B$ covers $B$, expanding with bounded distortion. The set $B'$ thus selected becomes 
an element of a partition and the argument is repeated for the points that were outside $B'$. However these returns have, in general, overlapping 
domains. This requires an additional effort to construct a countable family of expanding branches $f^{R_i}: B_i \to B$ with pairwise disjoint 
domains covering a full measure subset of $B$, assuring the integrability of the return times. \cite[Section 4.1]{alves.dias.luzzatto}.

We introduce in this paper an alternative induction scheme, motivated by results of \cite{sanchez.salas.2003}, 
giving an independent proof of \cite[Theorem 2, Theorem 3]{alves.dias.luzzatto}. Roughly speaking our induced maps are projections of the first 
return map of the natural extension $F: Z \to Z$ of $f$ to a suitable subset $\Sigma \subset Z$. Schematically, 
\begin{equation}\label{diagrama.fjss}
\begin{array}{ccc}
(Z,F) & \stackrel{\text{first return}}{\longrightarrow} & (\Sigma,F^{\bar{R}})\\
 \pi\downarrow &                                       & \downarrow \pi_{\Sigma}\\
 (M,f)  & \stackrel{\text{induction}}{\longrightarrow}        & (B,f^R)
\end{array} 
\end{equation}
See subsection \ref{subsec:estrategia}. 

Crucially, $\Sigma$ is chosen to be a \emph{'rectangle with regular returns'} (\textbf{RRR}). These geometrical structures were introduced 
in \cite{doob.2012} to study of ergodic properties of certain equilibrium states of non-uniformly expanding rational maps of the Riemann sphere 
proving, among other things, the existence of Markov structures with a non-exceptional conformal measure as reference.

Although \cite{doob.2012} is mostly based on geometrical and analytical properties of conformal maps, 
the proof of the existence of a \textbf{RRR} essentially relies upon smooth ergodic theory methods. This suggests the possibility to extend its 
construction for a more general class of non conformal mappings. This is precisely what we do for non-uniformly expanding maps with 
non-degenerate criticalities and/or singularities.

We should mention too \cite{bruin.1995} as an immediate antecedent to \cite{sanchez.salas.2003}. There the question was raised as to whether 
the integrability of the return times of an induced Markov map is a necessary condition for the existence of an \textbf{acim} $\mu$. It is 
proved that this is indeed the case, for interval piecewise monotonic maps preserving a conservative \textbf{acim} with positive entropy, 
whenever the induced Markov map satisfies some natural condition.

The induction method proposed in \cite{bruin.1995} can be summarized in the following diagram:
\begin{equation}\label{diagrama.bruin}
\begin{array}{ccc}
(\hat{I},\hat{f}) & \stackrel{\text{first return}}{\longrightarrow} & (\hat{J},\hat{f}^R)\\
 \pi\downarrow &                                       & \downarrow \pi \\
 (I,f)  & \stackrel{\text{induction}}{\longrightarrow}        & (J,f^R)
\end{array} 
\end{equation}
Here $\hat{f}: \hat{I} \to \hat{I}$ is a \emph{Markov extension} of $f : I \to I$, a construction due to Keller and Hofbauer. 
$\hat{f}: \hat{I} \to \hat{I}$ is a Markov transformation of a non-compact metric space $\hat{I}$ endowed with a naturally-defined countable 
partition defined using the monotonicity intervals of $f$. Moreover $\hat{I}$ comes endowed with a continuous and surjective map $\pi: \hat{I} \to I$ 
commuting with $f$ and $\hat{f}$, that is, $f \circ \pi = \pi \circ \hat{f}$. It is proved that if $f^R : J \to J$ 
is \emph{natural or naturally extendible} then there exists an interval $\hat{J} \subset \hat{I}$ making this diagram commutative. 
An induced Markov map is natural or naturally extendible if the return times are as small as possible and the branch-domains are as large as 
possible. See \cite[Lemma 2]{bruin.1995}. If $f$ preserves a conservative \textbf{acim} then the return times of a natural or naturally 
extendible induced Markov map is Lebesgue integrable. See \cite[Theorem 1]{bruin.1995}. 

This induction approach has received much attention recently in the work of Pesin, Senti and Zhang on the thermodynamics of non-uniform 
hyperbolic systems. See \cite{pesin.senti.2005}, \cite{pesin.senti.2008}, \cite{pesin.senti.zhang.2008}, \cite{pesin.senti.zhang.2016}.

The induction method presented in this paper has some advantages over \cite{bruin.1995} approach, at least for smooth non-uniformly expanding maps.
To begin with the phase space of the natural extension is compact, in contrast with Keller-Hofbauer's construction. Moreover, every measure $f$-invariant $\mu$ admits a (unique!) 
lift $\bar\mu$ to $Z$ meanwhile non every $f$-invariant Borel probability admits an elevation to a $\hat{f}$-invariant measure in the Markov 
extension. Moreover, our induced maps may or may not be natural or naturally extendible, a technical condition required to carry on the 
induction process under \cite{bruin.1995} premises.

Finally, let us underline that the present approach can be used to prove the existence of Markov structures for maps preserving a singular 
(wrt Lebesgue) $f$-invariant Borel probability $\mu$ with positive entropy and all its Lyapunov exponents positive. The question here is the 
choice of a good reference measure. We will discuss this problem in a forthcoming paper \cite{cappett.sanchez.2016}.

\subsection{Statement of main result}\label{subsec:enunciados}

\begin{definition}
We will say that a probability of Borel $f$-invariant $\mu$ is \emph{expanding} if its Lyapunov exponents 
\begin{equation}\label{lyapunov.exponents}
\chi ^+(x,v) = \lim_{n \to +\infty}\dfrac{1}{n}\log\|Df^n(x)v\|
\end{equation}
exists and they are positive for every $v \in T_xM - \{0\}$ and for $\mu$-a.e. $x \in M$.
\end{definition}

\begin{definition}\label{definicion.torre.markov}
Let $B \subset M$ an open simply connected subset. We say that $f^R: \bigcup_iB_i \to B$ is an 
\emph{induced expanding Markov map with bounded distortion and return times $R$} if 
$\{B_i \subset B\}$ is a collection of pairwise disjoint open subsets covering $B$ up-to zero volume, such that $R|B_i = R_i$ and 
$f^{R_i}: B_i \to B$ are $C^2$ expanding diffeomorphisms with bounded distortion, for every $i > 0$. That is, there exists constants 
$\lambda > 1$ and $C>0$ such that
\begin{equation}\label{rama.expansora}
  \|Df^{R_i}(x)v\| \geq \lambda\|v\|, \quad\forall \ v \in T_xM, \ \forall \ x \in {B}_i;
 \end{equation}
 \begin{equation}\label{distorsion.limitada}
  \dfrac{Jf^{R_i}(x)}{Jf^{R_i}(y)} \leq \exp(Cd(f^{R_i}(x),f^{R_i}(y))), \quad\forall \ x,y \in {B}_i, \  \forall \ i > 0,
 \end{equation}
 where $Jf = |\det(Df)|$ is the Jacobian of $f$ with respect to the Riemannian volume. We say that \emph{$f^R: \bigcup_iB_i \to B$ has an 
 integrable return times} if 
 \begin{equation}\label{retorno.integrable}
  \int{R}d\Vol = \sum_iR_i\Vol(B_i) < +\infty.
 \end{equation}
\end{definition}

We shall refer to an induced expanding Markov map $f^R: \bigcup_iB_i \to B$ with bounded distortion and integrable return times as a 
\emph{GMY structure with $\Vol$ as reference measure}.

Let $f^R: \bigcup_iB_i \to B$ be a GMY structure. By a folklore theorem there exists a unique ergodic, expanding 
$f^R$-invariant \textbf{acim} $\mu^*$ with a density bounded away from zero and $\infty$. In particular, $R$ is integrable wrt $\mu^*$, i.e. 
$\int{R}d\mu^* < +\infty$. We say that \emph{a $f$-invariant Borel probability $\mu$ is coinduced or generated} by $f^R: \bigcup_iB_i \to B$ if
$$
\mu = \dfrac{\sum_i\sum_{j=0}^{R_i-1}\mu^* \circ f^j}{\sum_iR_i\mu^*(B_i)}.
$$
It that case $\mu$ is expanding and ergodic, as it is easy to check. See \cite[Section 2.9]{yoccoz} or \cite[Lemma 3.1]{demelo}.

The following integrability condition is widely used in smooth ergodic theory and plays a crucial role in our approach.

\begin{definition}
We will say that a Borel $f$-invariant probability is \emph{regular} if 
\begin{equation}\label{hipotesis.R.1}
\int\max\{\log\|Df(x)\|, \log\|Df(x)^{-1}\|\}d\mu(x) < +\infty.
\end{equation}
\end{definition}
\ 
\\
\\
Next are the non-degeneracy conditions for critical and/or singular points.
\ 
\\
\\
\textbf{Non-degenerate critical points} \ \ \emph{We will say that 
$$
\mathcal{C} = \{x \in M: \text{ker}Df(x) \neq 0\},
$$
the critical point set, is \emph{not degenerate} if they exist $\beta, A, k > 0$ and an open neighborhood $V$ of $\mathcal{C}$ such that:} 
\begin{equation}\label{hipotesis.R.2}
 0 < \sup_{x \in V}\left|\dfrac{\|Df(x)^{-1}\|^{-1}}{d(x,\mathcal{C})^{\beta}}\right| < +\infty
\end{equation}
\begin{equation}\label{hipotesis.R.3}
\left|\log|\det(Df(x))| - \log|\det(Df(y))|\right| \leq A\frac{d(x,y)}{d(x,\textit{C})^{k}}, \quad\forall \ x,y \in M. 
\end{equation}
\ 
\\
\\
\textbf{Non-degenerate singular points} \ \emph{If $f$ has singularities 
$$
\mathcal{S} = \{x : \|Df(x)\| = +\infty \ \text{or is not defined}\}
$$
we will say that these are \emph{non-degenerate} if (\ref{hipotesis.R.3}) holds and, instead (\ref{hipotesis.R.2}),
\begin{equation}\label{hipotesis.R.4}
 0 < \sup_{x \in W}\left|\dfrac{\|Df(x)\|}{d(x,\mathcal{S})^{-\beta}}\right| < +\infty,
\end{equation}
where $W$ is an open neighborhood of $\mathcal{S}$.}
\ 
\\

The next is our main result.

\begin{theorem}\label{main.0}
Let $f : M \to M$ be a $C^2$ transformation with non-degenerate criticalities and/or singularities that preserves an ergodic, expanding, 
regular \textbf{acim} $\mu$. Then there exists a GMY structure $f^R: \bigcup_iB_i \to B$ generating $\mu$.
\end{theorem}
\ 
\\
Critical and singular cases were discussed separately in \cite{alves.dias.luzzatto}. In particular, the critical case 
\cite[Theorem 2]{alves.dias.luzzatto} does not need $\mu$ to be regular, in contrast with theorem \ref{main.0}. Actually, our induction 
approach depends heavily on the existence and expanding properties of local unstable manifolds. This requires some regularity conditions on 
$\mu$ aimed at taking advantage of standard smooth ergodic theory tools, as Pesin's $\epsilon$-reduction theorem, Lyapunov charts and so on. 
This allows for a unified treatment of the critical and singular cases. On the other hand, we do not require $\log\|Df(x)\|$ to be locally 
Lipschitz.

\subsection{Strategy of the proof: the induction scheme}\label{subsec:estrategia}

Let us start by recalling the natural extension of a non-invertible map. The phase space is the inverse 
limit of $f:M\to M$: 
$$
Z = \{(z_n)_{n \geq 0}: f(z_{n+1}) = z_n\},
$$
i.e., points $z \in Z$ are simply \emph{pre-orbits} of $f$. $Z$ is a compact metric space with distance $d(z,z') = \sum_{n=0}^{+\infty}d_M(z_n,z_n')/2^n$, 
where $d_M$ is the natural metric of $M$ as the Riemannn space. $F : Z \to Z$ defined as
$$
F(z) = (f(z_0), z_0, \cdots)
$$
is an invertible transformation, with inverse $F^{-1}(z) = (z_1, z_2, \cdots )$ which extends to $f$, namely $f \circ \pi = \pi \circ F$, 
where $\pi : Z \to M$ is the projection on the first coordinate. If $f$ is continuous, then $F$ is a homeomorphism of $Z$ itself. 

Given a probability of Borel $\mu$ in $M$ there is a unique $\bar{\mu}$ in $Z$ such that $\bar\mu(\pi^{-1}_n(B)) = \mu(B)$, for all Borel 
subset $B \subset M$ and for every $n \geq 0$, where $\pi_n : Z \to M$ Is the projection on the $n$-th coordinate: 
$\pi_n(z_0,z_1,z_2,\cdots) = z_n$. $\bar{\mu}$ is $F$-invariant (resp. ergodic) if $\mu$ is $f$-invariant (resp. ergodic). 

\begin{definition}
$F:(Z,{\mathcal B}, \bar{\mu}) \to (Z,{\mathcal B}, \bar{\mu})$ is the \emph{natural extension} of $(f,\mu)$. 
\end{definition}

$\pi : Z \to M$ is a Cantor fiber bundle, that is, every fiber $\mathcal{F}(z) = \pi^{-1}(\pi(z))$ is a Cantor set. If $f$ is uniformly 
expanding these fibers vary continuously and coherently so that the fiber bundle is locally trivial: for every $x \in M$ there exists an 
open neighborhood $U_x$ such that $\pi^{-1}(U_x)$ is homeomorphic to the cartesian product $U \times \mathcal{F}$, where $\mathcal{F}$ is a typical 'vertical' 
fiber of $Z$. This gives to $Z$ a structure of continuous lamination. Every local plaque $U_z \subset Z$ passing by $z$ is a local unstable 
manifold, that is, the backward orbit under $F$ of every $w \in U_z$ is exponentially asymptotic to the backward orbit of $z$.

When $f$ is non-uniformly expanding we get instead a measurable lamination of local unstable manifolds, that is, $\pi^{-1}(U_x)$ is no longer 
homeomorphic but measurable isomorphic to $U \times K$, that is, there exist an invertible measurable map 
$\varphi: \pi^{-1}(U_x) \to U \times \mathcal{F}$ with measurable inverse. Nevertheless, $Z$ is covered by 'flow-boxes'  
$\Sigma$ which are \emph{homeomorphic} to cartesian products $B \times K$ of an open set $B \subset M$ times a Cantor subset 
$K \subset \mathcal{F}$. Furthermore, $\pi|\Sigma_z: \Sigma_z \to B$ is a diffeomorphism for every $z \in \Sigma$, being 
$\Sigma_z \subset \Sigma$ the plaque of 'height' $z$ in $\Sigma$. We will call \emph{rectangles} to these 'flow-boxes' and say that a map is 
\emph{laminated} if it preserves this structure.

As we mention before, our induced Markov maps are suitable projections of the first return map $F^{\overline{R}}: \Sigma \to \Sigma$ to a 
\emph{rectangle with regular returns} $\Sigma$. Additionally to a uniform control on the rate of contraction 
along inverse branches and good control on the forward and backward distortion of horizontal plaques of $\Sigma$, \textbf{RRR} has the 
crucial property that, for every $n,m\geq 0$ and for every $z,z'\in \Sigma$, the preimages $F^{-n}\Sigma_z$, $F^{-m}\Sigma_{z'}$ and its 
projections $\pi{F^{-n}\Sigma_z}$, $\pi{F^{-m}\Sigma_{z'}}$ are \emph{disjoint or nested}, similarly to the 'nice sets' appearing in the study of 
rational maps of the Riemann sphere, from which Yoccoz's puzzles are prominent and well known examples. See \cite{doob.2012}, 
\cite{przyticky.rivera-letelier.2007} and \cite{przyticky.rivera-letelier.2011}.

We will prove in subsection \ref{subsec:rectangulo.retornos.regulares} the existence of \textbf{RRR} with positive measure 
for non-uniformly expanding transformations of compact Riemannian manifolds with non-degenerate criticalities and/or singularities.

The 'disjoint or nested' property of \textbf{RRR} gives an elegant solution to the overlapping of expanding branches which appear 
naturally in the induction process, allowing for a great simplification of the construction of a GMY structure.

Our crucial observation is that the first return map to a \textbf{RRR} is a piecewise hyperbolic map. That is, there exists a decomposition 
into countably many rectangles $S_j$ which covers $\Sigma$ up-to a $\bar\mu$-zero set and laminated diffeomorphisms 
$F^{\overline{R}_j}: S_j \to U_j$ which maps hyperbolically $S_j$ onto 'horizontal' strips $U_j$ crossing $\Sigma$ such that 
$F^{\overline{R}}|S_j = F^{\overline{R}_j}$. Here hyperbolic means, as usual, that it expands (resp. contracts) uniformly the 'horizontal' 
(resp. 'vertical') plaques (resp. Cantor fibers) preserving the laminated structure of $\Sigma$. This is our main technical lemma 
\ref{lema.tecnico.principal.1}. 

This structure is replicated by iterations of the first return map. So we take a suitable iterate such that the corresponding hyperbolic 
branches $F^{\overline{R}_j}: S_j \to U_j$ projects onto extendable expanding branches $f^{\overline{R}_j}: B_j \to B$ with good distortion 
properties. One get thus an induced Markov map $f^R: \bigcup_iB_i \to B$ with integrable return times $R$ choosing the $S_j$'s which 
intersects a suitable unstable plaque $\Sigma_z$ and then projecting the corresponding branches onto $M$.

The \emph{'disjoint or nested'} property of \textbf{RRR} has several advantages. In first place, it gives an elegant solution to the 
overlapping of expanding branches domains appearing in this type of constructions. On the other hand, it provides a 
partial ordering to the family of expanding branches, proving therefore that there are maximal induced maps. These are precisely the induced 
Markov maps appearing in \cite{doob.2012}. However, in contrast with Doob's approach, \emph{our induced maps are not necessarily maximal}. Moreover, 
even though the maximal induced Markov maps associated to non-uniformly expanding rational maps of the Riemann sphere are natural and naturally 
extendible, this might not be the case in our setting, due to the non conformal character of $f$. See section \ref{sec:demostracion.main.0}.

Finally, as we shall see in section \ref{sec:herraduras}, the first return map $F^{\bar{R}}: \Sigma \itself$ to a \textbf{RRR} is a 
\emph{``horseshoe with infinitely many branches''} but, even though 'vertical rectangles' are mapped 
hyperbolically onto 'horizontal rectangles' they do not return at the same time. This correct a gap in the proof and statement of 
\cite[Lemma 4.1]{sanchez.salas.2003}.

\subsection{Overview}
In \emph{section \ref{sec:preliminares}} we collect some technical preliminary results which will be used in the proof of theorem \ref{main.0}. 
In \emph{subsection \ref{subsec:funciones.temperadas}} we prove some technical lemmas on slowly varying $\epsilon$-tempered functions. In 
\emph{subsection \ref{subsec:approach.critical.set}} we use non-degeneracy conditions and regularity of $\mu$ to show that preorbits approaches 
the critical/singular set with subexponential speed. This is important in the local unstable manifolds existence theorem. 
\emph{Subsection \ref{subsec:tma.variedad.inestable}} is dedicated to state and outline the proof of the local unstable manifold theorem for 
the non-uniformly expanding maps in consideration. \emph{Subsection \ref{subsec:distorsion.temperada}} establishes a tempered estimate for 
distortion of forward and backward iterates of local unstable manifolds, using the non-degeneracy conditions. In 
\emph{subsection \ref{subsec:rectangulo.retornos.regulares}} we prove the existence of rectangles with regular returns. \emph{In subsection 
\ref{subsec:densidades.condicionales}} we construct an increasing generating measurable partition $\xi$ of $Z$ called Pesin partition and 
recall the densities of conditional measures of $\bar\mu$ wrt $\xi$. Then we use Rochlin's decomposition theorem to prove that the push-forward 
$\pi^*\bar\mu_{\Sigma}$ of the restriction of $\bar\mu$ to $\Sigma$ is equivalent to the Riemannian volume. This is a crucial step in 
the proof of theorem \ref{main.0}. In \emph{section \ref{sec:demostracion.main.0}} we prove main technical lemma \ref{lema.tecnico.principal.1} 
and, consequently, theorem \ref{main.0}. Then we prove that return times of maximal induced maps is integrable. Finally, in
\emph{section \ref{sec:herraduras}} we discuss the structure of 'horseshoe with infinitely many branches' of the first return map to $\Sigma$.

\tableofcontents

\section{ Preliminary technical results}\label{sec:preliminares}
 
 \subsection{Tempered functions}\label{subsec:funciones.temperadas}
\begin{definition}
We say that a Borel function $\phi> 0$ is $\epsilon $-tempered for $F$ if
\begin{equation}\label{temperada}
 e^{-\epsilon} \leq \dfrac{\phi(F(z))}{\phi(z)} \leq e^{\epsilon}
\end{equation}
\end{definition}

\begin{lemma}\label{lema.kernel.temperado} 
Let $\phi: Z \to (0, + \infty) $ be a Borel function such that
$$
\liminf_{n \to \pm\infty}\dfrac{\log\phi(F^n(z))}{n} = 0 \quad(\ resp. \limsup \ ) \ \bar\mu-c.t.p. \ z \in Z_0.
$$
Then, for every $\epsilon> 0$ small, there are $\epsilon$-tempered Borel functions $\underline\phi_{\epsilon} > 0$
(resp.$\bar\phi_{\epsilon} > 0$) such that
\begin{equation}\label{kernel.temperado}
\underline\phi_{\epsilon} \leq \phi \quad(\ resp. \phi \leq \bar\phi_{\epsilon} \ ) \ \bar\mu-c.t.p. \ z \in Z_0.
\end{equation}
\end{lemma}

\begin{proof}
For every $ \epsilon> 0$ small and $ \bar \mu $-a.e. $ Z \in Z $ there exists $N = N (\epsilon, z)$ such that
$$
e^{-\epsilon|n|} \leq \inf_{|k| \geq n}\phi(F^n(z)) \leq e^{\epsilon|n|} \quad\forall \ n \geq N. 
$$
Therefore, we define
$$
\underline\phi_{\epsilon}(z) = \inf_{n \in \enteros}\phi(F^n(z))e^{|n|\epsilon}
$$(\ref{expansion.uniforme.1}), 
(\ref{forward.distorsion.limitada.1}) and (\ref{backward.distorsion.limitada.1})
This is a well defined and a positive Borel function. Then,
\begin{eqnarray*}
 \underline\phi_{\epsilon}(F(z)) & = & \inf_{n \in \enteros}\phi(F^{n+1}(z))e^{|n|\epsilon}\\
                                  & = & \inf_{m \in \enteros}\phi(F^{m}(z))e^{|m-1|\epsilon}.
\end{eqnarray*}
as $|m|-1 \leq |m-1| \leq |m|+1$, $\forall \ m \in \enteros$, clearly we have
$$
e^{-\epsilon}\underline\phi_{\epsilon}(z) \leq \underline\phi_{\epsilon}(F(z)) \leq e^{\epsilon}\underline\phi_{\epsilon}(z)
$$
Similarly, if
$$
\limsup_{n \to \pm\infty}\dfrac{\log\phi(F^n(z))}{n} = 0, \quad\bar\mu-c.t.p. \ z \in Z_0
$$
we have that for almost every $ z \in Z_0 $ there is $ N = N (\epsilon, z) $ such that
$$
e^{-\epsilon|n|} \leq \sup_{|k| \geq n}\phi(F^n(z)) \leq e^{\epsilon|n|} \quad\forall \ n \geq N. 
$$
Then we define
$$
\bar\phi_{\epsilon}(z) = \sup_{n \in \enteros}\phi(F^n(z))e^{-|n|\epsilon}
$$
and we prove in a similar manner that $\bar\phi$ is an $ \epsilon $-tempered positive Borel function.
\end{proof}

As an immediate corollary we get that, if 
$$
\lim_{n \to \pm\infty}\dfrac{\log\phi(F^n(z))}{n} = 0, \quad\bar\mu-c.t.p. \ z \in Z_0,
$$
then there are $\epsilon$-tempered positive Borel functions $\underline{\phi}_{\epsilon}$  and $\bar{\phi}_{\epsilon}$ such that
$$
\underline{\phi}_{\epsilon} \leq \phi \leq \bar\phi_{\epsilon}.
$$
\subsection{Approach to the critical set}\label{subsec:approach.critical.set}
From now on, we denote $\mathcal{C}$ the critical and/or singular point set satisfying hypothesis 
(\ref{hipotesis.R.2}) and (\ref{hipotesis.R.4}).

The next lemma proves that, generally, the pre-orbits approach with subexponential speed to the critical/singular set $\mathcal{C}$.
\begin{lemma}\label{lema.acercamiento.conjunto.critico}
Let
$$
\delta_{\mathcal{C}}(z) := d(\pi(z),{\mathcal{C}}).
$$
then,
\begin{equation}\label{acercamiento.conjunto.critico.0}
\lim_{n \to \pm\infty}\dfrac{\log\delta_{\mathcal{C}}(F^n(z))}{|n|} = 0, \quad\bar\mu-c.t.p. \ z \in Z_0. 
\end{equation}
\end{lemma}
\begin{proof}
Follow from Birkhoff's theorem since $\delta_{\mathcal{C}} = \delta_{\mathcal{C}}(z)$ is $\bar\mu$-integrable by the regularity of $\mu$ and 
hypothesis (\ref{hipotesis.R.2})/(\ref{hipotesis.R.4}).
\end{proof}

\subsection{Local unstable manifold theorem}\label{subsec:tma.variedad.inestable}

\begin{theorem}\label{variedad.inestable.local}
Let $f$ be a $C^2$ transformation with non degenerate critical/singu\-lar set $\mathcal{C}$ leaving invariant a regular ergodic expanding 
Borel probability $\mu$. Then for every $\epsilon > 0$ there exists $\epsilon$-tempered Borel functions 
$\alpha, \beta, \gamma:Z \to (0,+\infty)$ such that $\bar\mu$-a.e.:
\begin{enumerate}
\item the local unstable set
\begin{equation}\label{backward.asintotics}
 {W}_{\alpha}^{u}(z)=\{w\in Z: d(z_n,w_n)\leq \beta(z)e^{-n\chi/2}d(z_0,w_0) \quad \forall n\geq 0\}
 \end{equation}
 is a regular submanifold embedded in $Z$;
 \item $\pi | {W}^u_{\alpha}(z) : {W}^u_{\alpha}(z) \to B(\pi(z),\alpha(z)) \subset M$ is a diffeomorphism;
 \item $B(\pi(z),\alpha(z)) \subset B(\pi(z), \delta_{\mathcal{C}}(z)/2)$; in particular $B(\pi(z),\alpha(z))$ does not intersect the 
 critical/singular set $\mathcal{C}$;
 \item the family of local unstable manifolds is invariant, in that, $W_{\alpha}^{u}(F(z)) \subset F(W^u_{\alpha}(z))$ and 
 $f^n : \pi(F^{-n}W^u_{\alpha}(z)) \to W^u_{\alpha}(z)$ is an expanding diffeomorphism, modulated by $\gamma(z)$, i.e.
  \begin{equation}\label{expansion.uniforme}
\inf_{w \in F^{-n}[W^u_{\alpha}(z)]}\|(f^n)'(\pi(w))v\| \geq \gamma(z)e^{n\chi/2}\|v\|
\end{equation}
for every $v \in T_{\pi(w)}M$ and for every $n \geq 0$.
\end{enumerate}
\end{theorem}

This result is a well known result. See for example \cite[Proposition 14]{doob.2012}, \cite[Theorem 8]{ledrappier.1981}, 
\cite[Theorem 6.1]{ruelle.1979}, \- \cite{ruelle-shub.1980}, \cite[Theorem 1]{zhu.1998}. For the sake of completeness we outline a proof. 
Details will appear elsewhere.

\begin{proof}[Outline of a proof]

We shall prove that there are Borel functions $\alpha$ and $\beta$ with subexponential growth along the orbits of $F$ and $C^2$ 
diffeomorphisms
$$
\Phi_z : B(\pi(z),\alpha(z)) \to W^u_{\alpha}(z).
$$
Here $\Phi_z(x) = (f^{-n}_z(x))_{n \geq 0}$ where $f^{-n}_z : B(\pi(z),\alpha(z)) \to M$ is the (unique) local branch of 
$f^{-n}$ such that $f^{-n}_z(z_0) = z_{n}$, for every $n \geq 0$.

For this we first choose \emph{Lyapunov's charts} $\psi_z : B(0,\rho(z)) \to M$ for $\bar\mu$-a.e. $z \in Z$ such that, if 
$\bar{f}_z = \psi^{-1}_{F(z)} \circ f \circ \psi_z$ is the local representative of $f$ then
\begin{equation}\label{expansion.local.uniforme}
e^{\chi/2}\|u-v\| \leq \|\bar{f}_z(u)-\bar{f}_z(v)\| \leq e^{3/2\chi^+}\|u-v\|,
\end{equation}
for every $u,v \in B(0,\rho(z))$, where $\chi > 0$ (resp. $\chi^+$) are the least (resp. largest) Lyapunov exponent of $\mu$ and $\| \cdot \|$ 
is the norm defined by the standard inner product of $\real^m$.

Then $\bar{f}^{-n}_z = \psi^{-1}_{F^{-n}(z)} \circ f^{-n}_z \circ \psi_z$, the local representative of $f^{-n}_z$, is a contraction on a 
suitable domain, that is, 
\begin{equation}\label{contraccion.uniforme.ramas.inversas}
 e^{-3/2n\chi^+} \leq \|\bar{f}^{-n}_z(u) - \bar{f}^{-n}_z(v)\| \leq e^{-n\chi/2}\|u-v\|,
\end{equation}
for every $u,v \in B(0,\underline{\rho}_{\epsilon}(z))$, where $\underline{\rho}_{\epsilon} \leq \rho$ is a positive $\epsilon$-tempered Borel 
function provided by lemma \ref{lema.kernel.temperado}. 

The distortion between the standard metric of $\real^m$ and the Riemannian metric in $M$ due to the charts $\psi_z$ can be estimated as follows:
\begin{equation}\label{distorsion.lyapunov.charts}
Cd(\psi_z(x),\psi_z(y)) \leq \|x-y\| \leq D(z)d(\psi_z(x),\psi_z(y)), 
\end{equation}
for every $x,y \in B(0,\rho(z))$, where $C > 0$ is a suitable constant and $D(z)$ a positive Borel function with subexponential
growth along the orbits of $F$.

We use (\ref{distorsion.lyapunov.charts}) and lemma \ref{lema.kernel.temperado} to define 
\begin{equation}\label{definicion.alpha.0}
\alpha(z) := \min\left\{\dfrac{\underline{\rho}_{\epsilon/2}(z)}{\overline{D}_{\epsilon/2}(z)}, \underline{\delta}_{\epsilon}(z)\right\}
\end{equation}
where $\underline{\rho}_{\epsilon/2} \leq \rho$ and $\overline{D}_{\epsilon/2} \geq D$ are $\epsilon/2$-tempered functions and 
$\underline{\delta}_{\epsilon} \leq \delta_{\mathcal{C}}/2$ is an $\epsilon$-tempered function. Then, 
$B(\pi(z),\alpha(z)) \subset \psi_z(B(0,\underline{\rho}_{\epsilon}(z))$ and condition \emph{(3)} in the theorem \ref{variedad.inestable.local} holds true. 
To get estimative (\ref{backward.asintotics}) we use (\ref{contraccion.uniforme.ramas.inversas}) and (\ref{distorsion.lyapunov.charts}) and 
define
\begin{equation}\label{definicion.beta.0}
\beta(z) := \dfrac{\overline{D}_{\epsilon}(z)}{C} 
\end{equation}
where $\bar{D}_{\epsilon} \geq D$ is an $\epsilon$-tempered positive Borel function given by lemma \ref{lema.kernel.temperado}.

The construction of Lyapunov's charts is standard. See \cite{barreira.pesin} and \cite{katok.mendoza}, for instance. The first step is to 
define an invertible linear cocycle $\mathcal{L}$ which covers the natural extension $F : Z_0 \to Z_0$. $\mathcal{L}$ is the natural lift of the 
linear cocycle defined by the derivative $Df$. $Z_0$ is a totally $F$-invariant set of total measure, formed by the preorbites $z \in Z$, 
which do not intersect the critical/singular set $\mathcal{C}$.

By \emph{Pesin's $\epsilon$-reduction theorem} there exists a tempered linear change of coordinates, 
$C_{\epsilon}(z) : \real^m \to \real^m$ which transforms $\mathcal{L}$ into a uniformly expanding linear cycle $A_{\epsilon}(z)$ in $\real^m$ 
such that:
$$
e^{\chi-\epsilon}\|v\| \leq \|A_{\epsilon}(z)v\| \leq e^{\chi^+ + \epsilon}, \quad\bar\mu-a.e. \ z \in Z_0.
$$
To do this we define a \emph{Lyapunov metric} $|\cdot|^*_z$ by averaging the iterates of the linear cocycle $\mathcal{L}$ weighted with a 
suitable kernel. Then we get a linear isometry $C_{\epsilon}: \real^m \to T_{\pi(z)}M$ between $(\real^m, \| \cdot \|)$ and 
$(T_{\pi(z)}M, | \cdot |^*_z)$. See \cite[Theorem S.2.10]{katok.mendoza}.

Lyapunov's charts are the composition $\psi_z := \exp_{\pi(z)} \circ \, C_{\epsilon}(z)$ of the linear coordinate change $C_{\epsilon}(z)$ 
and a geodesic chart in $M$ defined at a (uniform) open ball in the tangent space $T_{\pi(z)}M$. Given $\epsilon > 0$ we choose 
$\rho(z) > 0$ in such way that the representative of $f$ in this local coordinate system is a $C^2$ nonlinear perturbation of the expanding 
linear map $A_{\epsilon}(z) : \real^m \to \real^m$ with $\epsilon$ small $C^1$ norm. This chart satisfies the distortion estimates 
(\ref{distorsion.lyapunov.charts}). See \cite[Theorem S.3.1]{katok.mendoza}. Then, for a well chosen and sufficiently small $\epsilon > 0$ the 
local representative $\bar{f}_z $ (resp. $\bar{f}^{-n}_z$) satisfies (\ref{expansion.local.uniforme}) (resp. \ref{contraccion.uniforme.ramas.inversas}).
\end{proof}

\begin{remark}
By (\ref{contraccion.uniforme.ramas.inversas}), (\ref{distorsion.lyapunov.charts}) and (\ref{definicion.beta.0}) it holds
\begin{equation}\label{expansion.local.f^n}
\dfrac{1}{\beta(z)}e^{n\chi/2}d(x,y) \leq d(f^n(x),f^n(y)) \leq \beta(z)e^{2n\chi^+}d(x,y), 
\end{equation}
$\forall \ x,y \in \pi(F^{-n}(W^u_{\alpha}(z)))$, for $\bar\mu$-a.e. $z \in Z$. Indeed, let $\epsilon > 0$ sufficiently small such that $3/2\chi^+ + \epsilon < 2\chi^+$ 
and $\bar{D}_{\epsilon} \geq D$ be an $\epsilon$-tempered function provided by lemma \ref{lema.kernel.temperado}. Then, for every 
$x,y \in \pi(F^{-n}(W^u_{\alpha}(z)))$
\begin{eqnarray*}
d(f^n(x),f^n(y)) & \leq & \dfrac{\bar{D}_{\epsilon}(F^{-n}(z))}{C}e^{3/2n\chi^+}d(x,y) \\
                     & \leq & \dfrac{\bar{D}_{\epsilon}(F^{-n}(z))}{C}e^{3/2n\chi^+}d(x,y)\\
                     & \leq & \dfrac{\bar{D}_{\epsilon}(z)}{C}e^{n(3/2\chi^++\epsilon)}d(x,y)\\
                     & \leq & \dfrac{\bar{D}_{\epsilon}(z)}{C}e^{2n\chi^+}d(x,y).
\end{eqnarray*}
Likewise the lower bound.
\end{remark}

\subsection{Tempered Distortion}\label{subsec:distorsion.temperada}  
Let us denote $JF(z) = Jf(\pi(z))$ the Jacobian of $F|W^u_{\alpha}(z)$, putting in $W^u_{\alpha}(z)$ the volume element obtained as a push-forward 
of the Riemannian volume element in $M$ by the parameterization $\Phi_z : B(\pi(z),\alpha(z)) \to W^u_{\alpha}(z)$.

\begin{lemma}\label{lema.distorsion.temperada}
There exists $\epsilon_0 > 0$ such that for every $0 < \epsilon < \epsilon_0$ there is an $\epsilon$-tempered Borel function 
$\gamma_d: Z \to (0,+\infty)$, such that
\begin{equation}\label{backward.distorsion.temperada}
\prod_{i=0}^{+\infty}\frac{JF(F^{-i}w)}{JF(F^{-i}w')} \leq \exp\{\gamma_d(z)d(w,w')\},
\end{equation}
for every $w,w'\in W^u_{\alpha}(z)$  and $n \geq 0$,
\begin{equation}\label{forward.distorsion.temperada}
\left\lvert \frac{JF^n(w)}{JF^n(w')}\right\rvert \leq \exp\{\gamma_d(z)d(f^n(\pi(w)),f^n(\pi(w')))\} 
\end{equation}
for every $w,w' \in F^{-n}W^u_{\alpha}(z)$. 
\end{lemma}

See \cite[Lemma 3.3]{sanchez.salas.2003}. In particular, for every $0 < k < n$, $w,w' \in F^{-n}W^u_{\alpha}(z)$:
$$
\left\lvert \frac{JF^k(w)}{JF^k(w')}\right\rvert \leq \exp\{\gamma_d(z)\beta(z)\lambda^{n-k}d(f^n(\pi(w)),f^n(\pi(w')))\}
$$
with $\lambda = e^{-\chi/3} \in (0,1)$.
\begin{proof}
By the hypothesis (\ref{hipotesis.R.3}) and the estimate (\ref{backward.asintotics}) in the theorem (\ref{variedad.inestable.local}), for every 
$w,w' \in F^{-n}W^u_{\alpha}(z)$:
\begin{eqnarray*}
\log\frac{JF^n(w)}{JF^n(w')} & \leq & \sum_{j=0}^{n-1}|\log{JF(F^j(w))} - \log{JF(F^j(w'))}|\\
                             & \leq & \sum_{j=0}^{n-1}A\dfrac{d_M(\pi(F^j(w)),\pi(F^j(w')))}{d_M(\pi(F^j(w)),\mathcal{C})^k}\\
                             & \leq & \sum_{j=0}^{n-1}2^kA\dfrac{\beta(z)e^{-(n-j)\chi/2}d_M(f^n(\pi(w)),f^n(\pi(w')))}{d_M(\pi(F^{-n+j}(z)),\mathcal{C})^k},
\end{eqnarray*}
because
$$
d_M(\pi(F^j(w)),\mathcal{C}) \geq d_M(\pi(F^{-n+j}(z)),\mathcal{C}) - \alpha(F^{-n+j}(z)) \geq 1/2d_M(\pi(F^{-n+j}(z)),\mathcal{C})
$$ 
since $\alpha(z) < 1/2\delta_{\mathcal{C}}(z)$ and $F^j(w) \in W^u_{\alpha}(F^{-n+j}(z))$.
By lemma \ref{lema.acercamiento.conjunto.critico} and lemma \ref{kernel.temperado}, for every $\epsilon > 0$ there exist $\epsilon/2k$-tempered functions 
$\delta^*(z) = \bar\delta_{\epsilon/2k}(z) \geq \delta_{\mathcal{C}}(z)$, for $\bar\mu$-a.e. Therefore,
\begin{eqnarray*}
 \dfrac{1}{d_M(\pi(F^{-n+j}(z)),\mathcal{C})} & \leq & \dfrac{1}{{\delta}^*(F^{-n+j}(z))}\\
                                           &  =   & \dfrac{1}{{\delta}^*(z)}\dfrac{{\delta}^*(z)}{{\delta}^*(F^{-n+j}(F^n(z))} \\
                                           & \leq & \dfrac{e^{(n-j)\epsilon/2k}}{{\delta}^*(z)}.
\end{eqnarray*}
Let $\bar\beta \geq \beta$ be an $\epsilon/2$-tempered Borel function given by lemma \ref{lema.kernel.temperado}. Then, choosing $\epsilon_0 > 0$ such 
that $\chi/3 < \chi/2 - \epsilon/2k$ whenever $0 < \epsilon < \epsilon_0$, we have
\begin{eqnarray*}
 \log\left\lvert\frac{JF^n(w)}{JF^n(w')}\right\rvert \leq \dfrac{2^kA\bar\beta(z)}{({\delta}^*(z))^k}d_M(f^n(\pi(w)),f^n(\pi(w')))\sum_{j=0}^{n-1}e^{-(n-j)\chi/2}e^{(n-j)\epsilon/2k}\\
 \leq \dfrac{2^kA\bar\beta(z)}{({\delta}^*(z))^k}d_M(f^n(\pi(w)),f^n(\pi(w')))\sum_{j=0}^{n-1}e^{-(n-j)(\chi/2-\epsilon/2k)}\\
 \leq \dfrac{2^kA\bar\beta(z)}{({\delta}^*(z))^k}d_M(f^n(\pi(w)),f^n(\pi(w')))\sum_{n=0}^{+\infty}e^{-n\chi/3}\\
 \leq \dfrac{2^kA}{1-e^{-\chi/3}}\dfrac{\bar\beta(z)}{({\delta}^*(z))^k}d_M(f^n(\pi(w)),f^n(\pi(w')))
 \end{eqnarray*}
Now, $\bar\beta$ is $\epsilon/2$-tempered and ${\delta}^*$ is $\epsilon/2k$-tempered then
$$
\gamma_d(z) := \dfrac{2^kA}{1-e^{-\chi/3}}\dfrac{\bar\beta(z)}{({\delta}^*(z))^k},
$$
is $\epsilon$-tempered, proving (\ref{forward.distorsion.temperada}).
Similarly,
\begin{eqnarray*}
\log\prod_{i=0}^{+\infty}\frac{JF(F^{-i}w)}{JF(F^{-i}w')} & \leq & \sum_{i=0}^{+\infty}2^kA\dfrac{\beta(z)e^{-i\chi/2}d_M(\pi(w),\pi(w'))}{d_M(\pi(F^{-i}(z),\mathcal{C})^k} \\
                                                          & \leq & \dfrac{2^kA\bar\beta(z)}{({\delta}^*(z))^k}d_M(\pi(w),\pi(w'))\sum_{i=0}^{+\infty}e^{-i(\chi/2-\epsilon/2k)}\\
                                                          & \leq & \dfrac{2^kA\bar\beta(z)}{({\delta}^*(z))^k}d_M(\pi(w),\pi(w'))\sum_{i=0}^{+\infty}e^{-i\chi/3}\\
                                                          &   =  & \dfrac{2^kA}{1-e^{-\chi/3}}\dfrac{\bar\beta(z)}{({\delta}^*(z))^k}d_M(\pi(w),\pi(w')),
\end{eqnarray*}
so proving (\ref{backward.distorsion.temperada}).
\end{proof}

\subsection{Rectangles with regular returns}\label{subsec:rectangulo.retornos.regulares}
By the non-uniform expanding character of the dynamics to $(f,\mu)$ $Z$ produces a measurable lamination by local unstable manifolds, that is, 
we can cover $Z$, up to a zero measure set, by \emph{'flow boxes'} which are homeomorphic to the cartesian product of a Cantor set and an open 
subset of $M$. We call \emph{rectangles} to these 'flow boxes'. Let us state this more precisely. For this we recall that 
$B \subset M$ is an \emph{open disc} if it is diffeomorphic to an open ball $B(x, r) \subset M$.

\begin{definition}
Let $B \subset M$ an open disc and $K \subset {\mathcal F}(z) := \pi^{-1}\pi(z)$ a non trivial Cantor subset. We say that $\Sigma \subset Z$ is 
a \emph{rectangle with base $B$ and height $K$} if there exists a homeomorphism $\Phi: \Sigma \to B \times K$ such that
$$
\forall \ w \in K: \quad \pi_w := \pi|\Sigma_w : \Sigma_w \mapto B, \ \text{is a diffeomorphism onto} \ B.
$$
where $\Sigma_w = \Phi^{-1}(B \times \{w\})$ is the 'plaque' of 'height' $w \in K$.

$\Sigma$ is a \emph{regular rectangle} if the cross sections $K_w := \pi^{-1}\pi(w) \cap \Sigma$ are homeomorphic to each other by holonomy, 
that is,
$$
\phi_{w,w'} : K_w \mapto K_{w'} \quad \text{is a homeomorphism $\forall \ w,w' \in \Sigma$},
$$
where $\phi_{w,w'}(w'') = \Sigma_{w''} \cap K_{w'}$, and $\Sigma_{w''}$ is the plaque containing $w'' \in K_w$.
\end{definition}

\begin{remark}
In particular, if $\Sigma$ it is regular, $\Phi : \Sigma \mapto B \times {K}$, defined as $\Phi(w) = (\pi(w),\phi_{w,z}(w))$ is a homeomorphism.
\end{remark}

\textbf{Important warning}: \emph{All the rectangles appearing in our construction are regular, so we shall omit to mention this in our 
statements, unless notice in contrary.}

\begin{remark}
It follows from the local unstable manifold theorem that $Z$ has a measurable laminated structure: for every $x \in M$ there exists 
an open ball $B_x \subset M$ such that there exists an invertible measurable map with measurable inverse, 
$\phi: \pi^{-1}(B_x) \to B_x \times \mathcal{F}$. Moreover, for every $n > 0$ there exists a compact set $Z_n \subset Z$ with $\bar\mu(Z_n) \geq 1 - 2^{-n}$ such that $Z_n$ is covered by 
rectangles $\Sigma_i$ which are domains of local charts $\Phi_i:\Sigma_i \to B_i \times K_i$, where $\Phi_i$ are homeomorphisms. For this we 
just take $Z_n$ a large compact set such that $\alpha$ and $z \mapsto W^u_{\alpha}(z)$ restricted to $Z_n$ are continuous.
\end{remark}

\begin{definition}\label{definicion.rectangulo.retornos.regulares}
We say that a rectangle $\Sigma \subset Z$ of base $B$ and fiber $K \subset \mathcal{F}(z)$ has the property of \emph{regular returns} if for 
every $z,z' \in \Sigma$ and for all $n,m > 0$:
\begin{enumerate}
\item $F^{-n}\Sigma_z$ and $F^{-m}\Sigma_{z'}$ (resp. $\pi{F^{-n}\Sigma_z}$ and $\pi{F^{-m}\Sigma_{z'}}$) are disjoint or nested;
\item $\pi|F^{-n}\Sigma_z: F^{-n}\Sigma_z \to \pi{F^{-n}\Sigma_z}$ is a diffeomorphism; 
\item for every $z \in \Sigma$ and every $n > 0$, $f^n : \pi{F^{-n}\Sigma_z} \to B$ is an expanding diffeomorphism with bounded volume 
distortion, that is, there exists $C>0$ such that for every $z \in \Sigma$: $f^n: \pi{F^{-n}\Sigma_z} \to \pi{\Sigma_z}$ is a diffeomorphism 
satisfying the following estimates:
\begin{equation}\label{expansion.uniforme.1}
\|Df^n(\pi(w))v\| \geq C^{-1}e^{n\chi/2}\|v\|,
\end{equation}
for every $v \in T_{\pi(w)}M$ and $w \in {F^{-n}\Sigma_z}$;
\begin{equation}\label{forward.distorsion.limitada.1}
\dfrac{Jf^n(\pi(w))}{Jf^n(\pi(w'))} \leq \exp(Cd(f^n(\pi(w)),f^n(\pi(w'))), \ \forall \ w,w' \in {F^{-n}\Sigma_z},
\end{equation}
for every $w,w' \in {F^{-n}\Sigma_z}$;
\begin{equation}\label{backward.distorsion.limitada.1}
\prod_{i=0}^{+\infty}\dfrac{Jf(\pi(F^{-i}w))}{Jf(\pi(F^{-i}w'))} \leq \exp(Cd(\pi(w),\pi(w'))),
\end{equation}
for every $w,w' \in {F^{-n}\Sigma_z}$.
\end{enumerate}
\end{definition}

See \cite[Definition 23]{doob.2012}.

\begin{lemma}\label{lema.cilindro.retornos.regulares}
Let $(f,\mu)$ be a dynamical system satisfying the hypotheses of the theorem \ref{main.0}. Then there exists a rectangle with regular returns 
$\Sigma \subset Z$ with positive measure.
\end{lemma}

Essentially this was proved in \cite[Theorem 24]{doob.2012}, so we simply outline main steps of the proof and refer to that paper for details.

\begin{proof}
The idea is to prove that there exists an open disc $B$ and a rectangle of positive measure $\Sigma$ over $B$ such that
\begin{equation}\label{non.intersecting.condition.1}
\forall \ z \in \Sigma: \quad \pi(F^{-n}W(z)) \cap \partial{B} = \emptyset, \quad\forall \ n \geq 0,
\end{equation}
for it follows easily from the above condition that for every $z,z' \in \Sigma$ and for every $n,m \geq 0$, 
$F^{-n}B_z$ (resp. $\pi(F^{-n}B_z)$) and $F^{-m}B_{z'}$ (resp. $\pi(F^{-m}B_{z'})$) are disjoint or nested, where 
$B_z = \pi^{-1}(B) \cap W(z)$ is a diffeomorphic copy of $B$ in the unstable plaque $W(z)$.

We start choosing a compact set $\Sigma_0$ with $\bar\mu(\Sigma_0) \geq 1 - \delta$ such that they are continuous the local unstable manifolds 
$W^u_{\alpha}$ and the functions $\alpha, \beta, \gamma$ and $\gamma_d$ in theorem \ref{variedad.inestable.local} and lemma 
\ref{lema.distorsion.temperada} vary continuously with $z \in \Sigma_0$. This is possible by Egorov-Lusin's theorem. 

In particular, there exists $C > 0$ such that for every $z \in \Sigma_0$:
\begin{equation}\label{expansion.uniforme.0}
\|Df^n(\pi(w))v\| \geq C^{-1}e^{n\chi/2}\|v\|
\end{equation}
for every $v \in T_{\pi(w)}M$ and $w \in F^{-n}W^u_{\alpha}(z)$;
\begin{equation}\label{forward.distorsion.limitada.0}
\dfrac{Jf^n(\pi(w))}{Jf^n(\pi(w'))} \leq \exp(Cd(f^n(\pi(w)),f^n(\pi(w'))), 
\end{equation}
for every $w,w' \in {F^{-n}W^u_{\alpha}(z)}$ y 
\begin{equation}\label{backward.distorsion.limitada.0}
\prod_{i=0}^{+\infty}\dfrac{Jf(\pi(F^{-i}w))}{Jf(\pi(F^{-i}w'))} \leq \exp(Cd(\pi(w),\pi(w'))), 
\end{equation}
for every $w,w' \in {F^{-n}W^u_{\alpha}(z)}$. For this it suffices to take $C = \max_{z \in \Sigma_0}\{\gamma(z)^{-1}, \gamma_d(z)\}$.

Let us denote $W(z) = W^u_{\alpha}(z)$ and let $\alpha_0 > 0$ be the minimum of $\alpha|\Sigma_0$. We choose $x \in M$ such that 
$$
\Sigma_1 = \pi^{-1}(B(x,\alpha_0/2)) \cap \bigcup_{z \in \Sigma_0}W(z)
$$
has positive measure. Then $\Sigma_1$ is a rectangle. 

Now we prove that there exists an open ball $B \subset B(x,\alpha_0/2)$ and rectangle $\Sigma_2 \subset \Sigma_1$ over $B$, with positive 
measure and $N_0 > 0$ such that (\ref{non.intersecting.condition.1}) holds for every $n \geq N_0$.
\ 
\\
\\
\textbf{First Claim}: \emph{there exists $\alpha_0/8 < r < \alpha_0/4$, an integer $N_0 > 0$ and a subset $\Sigma_{0,N_0} \subset \Sigma_0$ of 
large measure such that,
$$
\forall \ z \in \Sigma_{0,N_0}: \ \diam(\pi(F^{-n}W(z))) < d(\pi(F^{-n}(z)),\partial{B(x,r)})/2, \ \forall \ n \geq N_0.
$$
}

By \cite[Lemma 21]{doob.2012} $d(\pi(F^{-n}(z)),\partial{B(x,r)})$ decays subexponentially for $\bar\mu$-a.e. for Lebesgue almost every 
$r \in (\alpha_0/8,\alpha_0/4)$. Fix one such $r$. Then for any small $\epsilon > 0$ there exists a Borel function 
$\delta$ with subexponential growth such that
$$
d(\pi(F^{-n}(z)),\partial{B(x,r)}) \geq \delta(z)e^{-n\epsilon}.
$$
By the local unstable manifold theorem \ref{variedad.inestable.local} $\diam(\pi(F^{-n}W(z)))$ decays exponentially. Moreover, by 
(\ref{expansion.local.f^n}) and definition (\ref{definicion.beta.0}) of $\beta$, 
$$
\diam(\pi(F^{-n}W(z))) \leq \beta(z)\alpha(z)e^{-n\chi/2}, \quad\forall \ n > 0.
$$
To find $\Sigma_{0,N_0}$ we let $0 < \eta < \chi$ be sufficiently small, define $\epsilon = \chi/2 - \eta/2$, $\alpha_{max} = \max{\alpha|\Sigma_0}$, 
$\beta_{max} = \max{\beta|\Sigma_0}$ and observe that
$$
\diam(\pi(F^{-n}W(z))) < d(\pi(F^{-n}(z)),\partial{B(x,r)})/2 \quad\text{whenever}\quad \beta_{max}\alpha_{max} < \dfrac{\delta(z)e^{n\eta/2}}{2}.
$$
We define a non decreasing sequence of Borel sets 
$$
X(n) = \left\{z \in Z: \beta_{max}\alpha_{max} < \dfrac{\delta(z)e^{n\eta/2}}{2}\right\}.
$$
covering $Z$ up-to a $\bar\mu$-zero set. Let $\Sigma_{0,n} := \Sigma_0 \cap X(n)$ and 
choose $N_0 > 0$ sufficiently large such that $\bar\mu(\Sigma_{0,N_0}) \geq (1-\delta)\bar\mu(\Sigma_0)$. This proves the first claim. Therefore,
$$
\Sigma_2 = \pi^{-1}(B(x,r)) \cap \bigcup_{z \in \Sigma_{0,N_0}}W(z)
$$ 
is a rectangle $\Sigma_2 \subset \Sigma_1$ with positive measure such that
\begin{equation}\label{non.intersecting.condition.2}
\forall \ z \in \Sigma_2: \quad \pi(F^{-n}W(z)) \cap \partial{B(x,r)} = \emptyset, \quad\forall \ n \geq N_0.
\end{equation}

Now, notice that (\ref{non.intersecting.condition.2}) is equivalent to the following: if $n \geq 0$ and there exists $k > 0$ 
such that
\begin{equation}\label{non.intersecting.condition.3}
\pi(F^{-n-k}W(z)) \cap \partial\pi(F^{-n}{B'_z}) \not= \emptyset. 
\end{equation}
then $k \leq N_0$. Therefore, for each $z \in \Sigma_2$ there exists a non decreasing sequence of positive integers $\{n_j(z)\}$ such that, 
for every $j > 0$ there exists $0 < k_j \leq N_0$ with
$$
\pi(F^{-n_j(z)-k_j}W(z)) \cap \partial\pi(F^{-n_j(z)}{B'_z}) \not= \emptyset.
$$
The sequence $\{n_j(z)\}$ is either finite or infinite. 
\ 
\\
\\
\textbf{Second Claim}: \emph{$\{n_j(z)\}$ is finite $\bar\mu$-a.e. $z \in \Sigma_2$.} 
\ 
\\
\\
Indeed, let us suppose that $C \subset \Sigma_2$ is a Borel subset with positive measure such that $\{n_j(z)\}$ is infinite for every $z \in C$. 
By \cite[Lemma 22]{doob.2012} there exists a continuous positive function $\theta: [0,+\infty) \to [0,+\infty)$ with $\theta(0) = 0$ such that for 
every set $V$ such that $V \cap f^k(V) \not= \emptyset$ for some $0 < k \leq N$ then $\dist(V,Fix(f^{N!})) < \theta(\diam(V))$. As 
$\diam(\pi(F^{-n}W(z))) \to 0$ exponentially then
$$
\dist(\pi(F^{-n}W(z)),Fix(f^{N_0!})) \to 0,
$$
uniformly in $z \in C$. Hence, for every $\epsilon > 0$
$$
\mu(B(Fix(f^{N_0!}),\epsilon)) = \bar\mu(\pi^{-1}B(Fix(f^{N_0!}),\epsilon)) \geq \bar\mu(F^{-n}C) = \bar\mu(C) > 0.
$$
Therefore $\mu$ must be concentrated on a repelling periodic orbit contradicting that $\mu$ is absolutely continuous. This proves the 
Second Claim. Let $Y_N = \{z \in \Sigma_2: \max_jn_j(z) \leq N\}$. Then there exists $N_1 > 0$ such that 
$\bar\mu(Y_{N_1}) \geq (1-\delta)\bar\mu(\Sigma_2)$. Hence
$$
\Sigma_3 = \pi^{-1}(B') \cap \bigcup_{z \in Y_{N_1}}W(z)
$$
is a subrectangle of $\Sigma_2$ with $\bar\mu(\Sigma_3) \geq (1-\delta)\bar\mu(\Sigma_2)$ such that
\begin{equation}\label{non.intersecting.condition.4}
\forall \ z \in \Sigma_3: \ \pi(F^{-n-k}W(z)) \cap \partial\pi(F^{-n}{B'_z}) = \emptyset, \ \forall \ n \geq N_1, k > 0.
\end{equation}
Notice that $F^{-N_1}(\Sigma_3)$ is a rectangle satisfying (\ref{expansion.uniforme.1}), 
(\ref{forward.distorsion.limitada.1}) and (\ref{backward.distorsion.limitada.1}). Let $B$ be a connected component of $f^{-N_1}(B')$ 
such that
$$
\Sigma = \pi^{-1}(B) \cap F^{-N_1}(\Sigma_3)
$$
has positive measure. $\Sigma$ is a rectangle satisfying (\ref{expansion.uniforme.1}), (\ref{forward.distorsion.limitada.1}) and 
(\ref{backward.distorsion.limitada.1}) and (\ref{non.intersecting.condition.1}) and has therefore the regular returns property.
\end{proof}

\subsection{Pesin partitions and density of conditional measures}\label{subsec:densidades.condicionales}

As $\bar\mu(\Sigma) > 0$, by the ergodicity of $\mu$, there exists for $\bar\mu$-a.e. $z \in Z$ a first time of entry $e(z) > 0$
of the positive orbit of $z$ into $\Sigma$: 
$$
e(z) = \min\{n > 0: F^n(z) \in \Sigma, \ F^k(z) \not\in \Sigma, \ 0 \leq k < n\}
$$
We define a partition $\xi$ of $Z$ as 
$$
\xi(z) = F^{-e(z)}\Sigma_{F^{e(z)}(z)}, \ \text{if} \ z \not\in \Sigma \ \text{and} \ \xi(z) = \Sigma_z, \ \text{si} \ z \in \Sigma.
$$

\begin{lemma}\label{lema.particion.pesin}
\begin{enumerate}
\item $\xi,$ is a measurable partition of $Z$ subordinate to the unstable lamination of $Z$, this is $\xi(z)\subset W_{\alpha}^{u}(z)$, $\bar\mu$-a.e.;
\item $\xi$ is increasing: $F^{-1}\xi \geq \xi$: $F^{-1}(P) \cap Q \not= \emptyset$ implies $F^{-1}(P) \subset Q$, for every $P,Q \in \xi$;
\item $\xi$ generate the $\sigma$-algebra of Borelians of $Z$: $\bigvee_{n=0}^{+\infty}F^{-n}\xi = \epsilon$, where $\epsilon(x) = \{x\}$ is 
the largest partition of $Z$ into points; 
\item $\pi | \xi(z) : \xi(z) \to M$ is a diffeomorphism onto its image;
\item $F^{e(z)} : \xi(z) \to \Sigma_{F^{e(z)}(z)}$ is projected onto a diffeomorphism $f^{e(z)}: \pi\xi(z) \to B$ satisfying 
(\ref{expansion.uniforme.1}), (\ref{forward.distorsion.limitada.1}) and (\ref{backward.distorsion.limitada.1}).
\end{enumerate}
\end{lemma}

See \cite[Proposition 3.2]{ledrappier.1981}. Compare also \cite[Section 4 and Section 6]{doob.2012}.

\begin{definition}
$\xi$ is a \emph{Pesin partition} for $(f,\mu)$.
\end{definition}

Let $\xi(z)$ the atom of the Pesin partition $\xi$ which contains $z \in Z$. By Rochlin's theorem, there is a system of conditional measures 
$\bar\mu_{\xi(z)}$ such that 
$$
\bar\mu(A) = \int_{Z/\xi}\bar\mu_{\xi(z)}(A \cap \xi(z))d\bar\mu_{\xi}(\xi(z)),
$$
where $\bar\mu_{\xi}$ is the quotient measure on $X/\xi$, that is,
$$
\int_Z\phi(w)d\bar\mu(w) = \int_{Z/\xi}\left(\int_{\xi(z)}\phi(w)d\bar\mu_{\xi(z)}(w)\right)d\bar\mu_{\xi}(\xi(z)),
$$
for every Borel measurable real function $\phi$.

Lets denote $\Vol_{\xi(z)}$ the pullback of the Riemannian volume of $M$ to the atom $\xi(z)$ via the diffeomorphism $\pi|\xi(z) : \xi(z) \to M$, that is, 
$\Vol_{\xi(z)}(A) = \Vol(\Phi^{-1}_z(A)$, for all Borelian $A \subset \xi(z)$, where $\Phi_z : B(\pi(z),\alpha(z)) \to W^u_{\alpha}(z)$ is the 
parametrization of the local unstable manifolds $W^u_{\alpha}(z)$ containing $\xi(z)$.

Since $\mu$ is absolutely continuous with respect to the Riemannian volume, then it satisfies the Pesin-Rochlin formula 
$h(\mu) = \sum_{\chi_i(x) > 0}\chi_i(x)\dim{E_i}(x)d\mu(x)$. This allows to compute explicitly the density of the conditional measures 
$\bar\mu_{\xi(z)}$ with respect to the volume $\Vol_{\xi(z)}$. 

\begin{lemma}\label{lema.densidad.condicional.barmu}
\begin{equation}\label{densidad.condicional.barmu}
\bar\mu_{\xi(z)}(A) = \dfrac{\int_A\Delta(z,w)d\Vol_{\xi(z)}(w)}{\int_{\xi(z)}\Delta(z,w)d\Vol_{\xi(z)}(w)}, 
\end{equation}
for every Borel subset $A \subset \xi(z)$, where 
$$
\Delta(z,w) = \prod_{i=0}^{+\infty}\dfrac{Jf(\pi(F^{-i}z))}{Jf(\pi(F^{-i}w))}.
$$
\end{lemma} 
See \cite[Proposition 3.6]{ledrappier.1981} and \cite[Proposition 32]{doob.2012}.

Let $\pi^*\bar\mu_{\Sigma}(A) = \bar\mu_{\Sigma}(\pi^{-1}(A))$ be the push-forward of $\bar\mu_{\Sigma}$, the 
restriction $\bar\mu|\Sigma$ normalized to a probability. The following shall be used later.

\begin{lemma}\label{lema.equivalencia.barmu.Vol}
Let $\Sigma$ be a rectangle with regular returns. Then, there are constants $C_0,C_1 > 0$ such that, 
\begin{equation}\label{equivalencia.barmu.Vol}
C_0\Vol(A) \leq \pi^*\bar\mu_{\Sigma}(A) \leq C_1\Vol(A) \quad\text{for every Borel subset} \ A \subset B.
\end{equation}
\end{lemma}
\begin{proof}
Let $\{\bar\mu_{\xi(z)}\}$ be the Rochlin decomposition of $\bar\mu|\Sigma$ with respect to $\xi$: 
$$
\bar\mu(\pi^{-1}(A) \cap \Sigma) = \int_{\Sigma/\xi}\bar\mu_{\xi(z)}(\pi^{-1}(A))d\bar\mu_{\xi}(\xi(z)).
$$
Then, from (\ref{densidad.condicional.barmu}) in lemma \ref{lema.densidad.condicional.barmu} and (\ref{backward.distorsion.limitada.1}) in the definition 
of a rectangle with regular returns we have
$$
\dfrac{e^{-2C\diam(B)}\Vol(A)}{\Vol(B)} \leq \bar\mu_{\xi(z)}(\pi^{-1}(A)) \leq \dfrac{e^{2C\diam(B)}\Vol(A)}{\Vol(B)},
$$
for every Borel subset $A \subset B$. Therefore, using the Rochlin decomposition,
$$
C_0\Vol(A) \leq \bar\mu_{\Sigma}(\pi^{-1}(A)) \leq C_1\Vol(A),
$$
for every Borel subset $A \subset B$, where
$$
C_0 := \dfrac{e^{-2C\diam(B)}\bar\mu_{\xi}(\Sigma/\xi)}{\Vol(B)\bar\mu(\Sigma)} \quad\text{and}\quad C_1 :=  \dfrac{e^{2C\diam(B)}\bar\mu_{\xi}(\Sigma/\xi)}{\Vol(B)\bar\mu(\Sigma)}.
$$
\end{proof}

\section{Proof of theorem \ref{main.0}}\label{sec:demostracion.main.0}

Let $\Sigma \subset Z$ be a rectangle with regular returns and $\bar\mu(\Sigma) > 0$ and $\tau : \Sigma \to \Sigma$ the first return map, that is, 
$\tau(z) = F^{\overline{R}(z)}(z)$, where $\overline{R}(z) = \min\{n > 0: F^n(z) \in \Sigma\}$ is the first return times to $\Sigma$.

\begin{definition}
Let $\Sigma$ a rectangle. We say that a subset ${U} \subset \Sigma$ is a $u$-rectangle if is a rectangle and 
${U} \cap \Sigma_w = \Sigma_w$, for every $w \in {U}$.
\end{definition}

Briefly an $u$-rectangle ${U}$ is a \emph{'horizontal strip'} crossing $\Sigma$ from 'left to right'.

\begin{lemma}\emph{(Main technical lemma)}\label{lema.tecnico.principal.1}
There exists a decomposition of $\Sigma$ into subrectangles $S_j\subset \Sigma$ and $u$-rectangles $U_j \subset \Sigma$ and non negative 
integers $\overline{R}_j \in \enteros^+$ such that 
\begin{enumerate}
\item $\bar\mu(\Sigma - \bigcup_jS_j) = \bar\mu(\Sigma - \bigcup_jU_j) = 0$;
\item $F^{\overline{R}_j}: S_j\mapto U_j$ is a laminated diffeomorphism such that, $\tau|S_j = F^{\overline{R}_j}$ and
\begin{equation}\label{expansion.uniforme.2}
\|Df^{\overline{R}_j}(\pi(w))v\| \geq C^{-1}e^{\overline{R}_j\chi/3}\|v\|, \quad\forall \ v \in T_{\pi(w)}M, \ \forall \ w \in S_j(z)
\end{equation}
for every $z \in S_j$ and 
\begin{equation}\label{forward.distorsion.limitada.2}
\dfrac{Jf^{\overline{R}_j}(\pi(w))}{Jf^{\overline{R}_j}(\pi(w'))} \leq \exp(Cd(f^{\overline{R}_j}(\pi(w)),f^{\overline{R}_j}(\pi(w'))), \quad\forall \ w,w' \in S_j(z), \ z \in S_j,
\end{equation}
where $C > 0$ is the constant provided by lemma \ref{lema.cilindro.retornos.regulares} and $S_j(z) := S_j \cap \Sigma_z$ is the leave of level $z$ on the rectangle $R_j$;
\item for every $m > 0$, $\#\{j > 0 : \overline{R}_j = m \} < +\infty$.
\end{enumerate}
\end{lemma}

The next well known result will be used in the proof of theorem \ref{main.0}.
\ 
\\
\\
{\bf Kac's lemma} \ \emph{Let $T : (X,\mathcal{A},\mu) \to (X,\mathcal{A},\mu)$ be a measurable, measure preserving transformation of a probability space, 
$A \in \mathcal{A}$ a subset with $\mu(A) > 0$ and $R_A(x) = \inf\{n > 0: T^n(x) \in A\}$ the first return times to $A$. Then,
$$
\int_A{R}_Ad\mu = 1
$$
}
See \cite[Theorem 2.4.6]{petersen}.
\ 
\\
\begin{proof}[Proof of theorem \ref{main.0}]
\ 
\\
Let $\tau: \Sigma \to \Sigma$ the first return map to a rectangle with regular returns. Then, $\tau^n(z) = F^{\overline{R}_n(z)}(z)$, where
$$
\overline{R}_n(z) = \sum_{k=0}^{n-1}\overline{R}_0(\tau^k(z)), \quad\forall \ n \geq 0.
$$
\textbf{Claim}: for every $n > 0$ there exist decompositions $\mathcal{R}^n = \{S^n_j\}$ (resp. $\mathcal{U}^n = \{U^n_j\}$) of $\Sigma$ into subrectangles 
(resp. $u$-rectangles) and non-negative integers $\overline{R}^n_j$ such that $\overline{R}_n|S^n_j = \overline{R}^n_j$ and
$$
F^{\overline{R}_n}|S^n_j = F^{\overline{R}^n_j}: S^n_j \to U^n_j
$$
is a laminated diffeomorphism satisfying estimates (\ref{expansion.uniforme.2}) and (\ref{forward.distorsion.limitada.2}).

Indeed, by the main technical lemma, there exist collections of subrectangles $\mathcal{R}$ and $u$-rectangles $\mathcal{U}$ of first level and times of 
return $\overline{R}_j$ such that $\tau|S_j = F^{\overline{R}_j}: S_j \to U_j$ is a laminated diffeomorphism satisfying  (\ref{expansion.uniforme.2}) and 
(\ref{forward.distorsion.limitada.2}). This is the first step of an induction argument. Now, suppose that 
we have done the construction up to $n$-th step. For every $j > 0$, we label $S_{nij} := U_j \cap S^n_i$ the subrectangles 
obtained by intersecting $U_j$ with the subrectangles of $n$-th level in $\mathcal{R}^n$. The collection $\{S_{nij}\}$ covers $U_j$ up-to a 
$\bar\mu$-zero measure set. Then, we decompose $S^n_j$ into subrectangles $S^{n+1}_{ij} \subset S^n_j$ of $n+1$-th generation such that 
$\tau(S^{n+1}_{ij}) = S_{nij}$. Let $U^{n+1}_{ij} := \tau^n(S_{nij})$ be the $n$-th first return of $S_{nij}$. Then $U^{n+1}_{ij}$ is an $u$-rectangle 
contained in $U^n_i$ and $\tau^{n+1}(S^{n+1}_{ij}) = U^{n+1}_{ij}$ and 
$$
\overline{R}_{n+1}|S^{n+1}_{ij} = \overline{R}_j + \overline{R}^n_i := \overline{R}^{n+1}_{ij}.
$$ 
Estimates (\ref{expansion.uniforme.2}) and (\ref{forward.distorsion.limitada.2}) follows inmediatly from (\ref{expansion.uniforme.1}) and 
(\ref{forward.distorsion.limitada.1}) in definition \ref{definicion.rectangulo.retornos.regulares}.

Let $N > 0$ be the smallest positive integer such that 
\begin{equation}\label{tiempo.hiperbolico}
\lambda := C^{-1}e^{N\chi/3} > 1
\end{equation} 
and notice that $\overline{R}_N(z) \geq N$ for $\bar\mu$-a.e. $z \in \Sigma$. For every $j > 0$, 
$F^{\overline{R}^N_j}: S^N_j \to U^N_j$ projects onto a diffeomorphism covering $B$, that is,
$$
\begin{array}{ccc}
S^N_j         & \stackrel{F^{\overline{R}^N_j}}{\longrightarrow}& U^N_j\\
\pi\downarrow &                                                 & \downarrow \pi\\
B^{N}_j       & \stackrel{f^{\overline{R}^N_j}}{\longrightarrow}& B
\end{array}
$$
such that,
$$
\|Df^{\overline{R}^N_j}(x)v\| \geq \lambda\|v\|, \quad \forall \ v \in T_{x}M, \ \forall \ x \in B^{N}_j
$$
and 
$$
\dfrac{Jf^{\overline{R}^N_j}(x)}{Jf^{\overline{R}^N_j}(y)} \leq \exp(Cd(f^{\overline{R}^N_j}(x),f^{\overline{R}^N_j}(x)), \ \forall \ x,y \in B^{N}_j.
$$
In other words, $f^{\overline{R}^N_j}: B^{N}_j \to B$ is an expanding branch with bounded distortion. 

Notice that these branches are extendible: choose $z \in S_j$ and $w = F^{\overline{R}_j}(z) \in U_j$. The plaque 
$U_j(w) = U_j \cap \Sigma_w$ is contained in some local unstable manifold $W^u_{\alpha}(w')$ and therefore, the plaque of level $z$ in $S_j$,
$S_j(z) = S_j \cap \Sigma_z$ is contained in $F^{-\overline{R}_j}W^u_{\alpha}(w')$. By (\ref{expansion.uniforme.0}) and 
(\ref{forward.distorsion.limitada.0}), estimatives (\ref{expansion.uniforme.2}) and 
(\ref{forward.distorsion.limitada.2}) also holds for the extension $f^{\overline{R}_j}: \hat{B}_j \to \hat{B}$, where 
$$
\hat{B}_j = \pi(F^{-\overline{R}_j}W^u_{\alpha}(w')) \quad\text{and}\quad \hat{B} = \pi(W^u_{\alpha}(w')).
$$

From now on we drop the superscript $N$ on the return times $\overline{R}^N_j$, subrectangles $S^N_j$ and $u$-rectangles $U^N_j$ and denote 
$\mathcal{R} = \{S_j\}$ the family of rectangles of $N$-th generation. Notice that, by Kac's lemma
\begin{equation}\label{Kac.lemma.N}
\int_{\Sigma}\overline{R}d\bar\mu = N, 
\end{equation}
since $\overline{R} = \overline{R}_N$ is the return times of $\tau^N$ and $\bar\mu$ is 
$\tau$-invariant. Then, by by Rochlin's decomposition theorem,
$$
\int_{\Sigma}\overline{R}(w)d\bar\mu(w) = \int_{\Sigma/\xi}\left(\int_{\xi(z)}\overline{R}(w)d\bar\mu_{\xi(z)}(w)\right)\bar\mu(d\xi) < +\infty,
$$
consequently,
\begin{equation}\label{plaque.integrability}
\int_{\xi(z)}\overline{R}(w)d\bar\mu_{\xi(z)}(w) < +\infty,
\end{equation}
for $\bar\mu-\text{a.e.} \ z \in \Sigma$. Now we choose a family of rectangles
$$
\mathcal{R}(z) = \{S_i \in \mathcal{R}: S_i \cap \xi(z) \not= \emptyset\}.
$$
Moreover, for $\bar\mu$-a.e. $z \in \Sigma$, 
\begin{equation}\label{plaque.rectangles.covering}
\bar\mu_{\xi(z)}\left(\xi(z) - \bigcup_{S_j \in \mathcal{R}(z)}S_j\right) = 0.
\end{equation}
This follows from Rochlin's theorem since
$$
\bar\mu\left(\Sigma - \bigcup_jS_j\right) = \int_{\Sigma/\xi}\bar\mu_{\xi(z)}\left(\Sigma - \bigcup_jS_j\right)d\bar\mu_{\xi}(\xi(z)) = 0.
$$

Let $z$ be a point in $\Sigma$ satisfying (\ref{plaque.integrability}) and (\ref{plaque.rectangles.covering}) and let $f^{R_i}: B_i \to B$ be the family of 
branches so chosen. Then $f^R : \bigcup_iB_i \to B$ is an induced expanding Markov map with bounded distortion and integrable return map. Indeed, 
as we argue in lemma \ref{lema.equivalencia.barmu.Vol} there exists a constant $K_0 > 0$ such that $Vol_{\xi(z)}(A) \leq K_0\bar\mu_{\xi(z)}(A),$ 
for every Borel subset $A \subset \xi(z)$. Therefore,
\begin{eqnarray*}
 \Vol\left(B - \bigcup_iB_i\right) & = & Vol_{\xi(z)}\left(\xi(z) - \bigcup_{S_i \in \mathcal{R}(z)}S_i\right) \\
                           & \leq & K_0\bar\mu_{\xi(z)}\left(\xi(z) - \bigcup_{S_i \in \mathcal{R}(z)}S_i\right)\\
                           & = & 0.
\end{eqnarray*}
Let $R$ be the return times of the induced map $f^R : \bigcup_iB_i \to B$. Then $R(\pi(w)) = \overline{R}(w)$ for every $w \in \xi(z)$ and thus, 
by (\ref{plaque.integrability}) and lemma \ref{lema.equivalencia.barmu.Vol},
\begin{eqnarray*}
 \int_B{R}(x)d\Vol(x) & = & \int_{\xi(z)}\overline{R}(w)d\Vol_{\xi(z)}(w)\\
                            & \leq & K_0\int_{\xi(z)}\overline{R}(w)d\bar\mu_{\xi(z)}(w)\\
                            &  <   & +\infty.
\end{eqnarray*}
\end{proof}

The next arguments are due to Doob. See \cite[Proposition 34]{doob.2012}. The domains $\{\pi(S_j): S_j \in \mathcal{R}\}$ associatesd to 
$\tau^N : \Sigma \itself$ are partially ordered by inclusion, since $\pi(S_j)$ are disjoint or nested. Then there exists a maximal family 
$\{B_i\}$ of pairwise disjoint open domains and associated extendable expanding branches with bounded distorsion $f^{R_i} : B_i \to B$. We 
call such $f^{R}: \bigcup_iB_i \to B$ a \emph{maximal induced map}. Next proposition proves that a maximal induced map is a generating 
GMY structure for $\mu$.

\begin{proposition}\label{maximal.induced.map}
Let $f^R: \bigcup_iB_i \to B$ be a maximal induced map. Then, $\Vol\left(B - \bigcup_iB_i\right) = 0$ and $R$ is Lebesgue 
integrable. 
\end{proposition}
\begin{proof}
By lemma \ref{lema.equivalencia.barmu.Vol} $\Vol|B$ is equivalent to $\pi^*\bar\mu_{\Sigma}$. Then,
$$
\Vol\left(B - \bigcup_iB_i\right) = \bar\mu_{\Sigma}\left(\Sigma - \bigcup_i\pi^{-1}(B_i)\right) = 0,
$$
since, by the disjoint or nested property of the rectangle $\Sigma$, 
$$
\pi^{-1}(B_i) = \bigcup_{S_j \subset \pi^{-1}(B_i)}S_j,
$$
$\bar\mu$-modulo zero and therefore, by the maximality of $\{B_i\}$, $\bigcup_i\pi^{-1}(B_i)$ covers $\Sigma$, up-to a $\bar\mu$-measure zero set.

On the other side, by the expansion estimate (\ref{expansion.local.f^n}) and the choice of a compact subset $\Sigma_0$ in the construction of $\Sigma$, 
there exists $\beta_0 > 0$ such that
$$
\dfrac{1}{\beta_0}e^{\overline{R}_j\chi/2}d(x,y) \leq d(f^{\overline{R}_j}(x),f^{\overline{R}_j}(y)) \leq \beta_0e^{\overline{R}_j2\chi^+}d(x,y),
$$
$\forall \ x,y \in B_j$, for every expanding branch $f^{\overline{R}_j}: B_j \to B$, where $\overline{R}_j = \overline{R}_N|S_j$ and $B_j = \pi(S_j)$. 
Therefore,
$$
\dfrac{1}{\beta_0}e^{-\overline{R}_j2\chi^+}\diam(B) \leq \diam(B_j) \leq \beta_0e^{-\overline{R}_j\chi/2}.
$$
Thus, $B_j \subset B_i$ whenever
$$
\dfrac{1}{\beta_0}e^{-\overline{R}_j2\chi^+} \leq \beta_0e^{-\overline{R}_i\chi/2}.
$$
Hence, if $f^{R_i}: B_i \to B$ belongs to a maximal family, then
$$
R_i \leq \inf_{B_j \subset B_i}\dfrac{4\log(\beta_0) + 4\chi^+\overline{R}_j}{\chi}.
$$
Equivalently, let $R: B \to \enteros^+$ such that $R|B_i = R_i$, then
\begin{equation}\label{R.bounds}
R(x) \leq \inf_{z \in \Sigma: \pi(z) = x}\dfrac{4\log(\beta_0) + 4\chi^+\overline{R}_N(z)}{\chi}, \quad\forall \ x \in B_i,
\end{equation}
since $\pi^{-1}(B_i) = \bigcup_{B_j \subset B_i}S_j$, up-to a $\bar\mu$-zero measure set, by the 'disjoint or nested' property. 
Thus, by lemma \ref{lema.equivalencia.barmu.Vol},
\begin{eqnarray*}
\sum_iR_i\Vol(B_i) &   =  & \int_B{R}(x)d\Vol(x)\\
                   & \leq & C_0^{-1}\int_B{R}(x)d\pi^*\bar\mu_{\Sigma}(x)\\
                   & =    & C_0^{-1}\int_{\Sigma}{R}(\pi(z))d\bar\mu_{\Sigma}(z)\\
                   & \leq & C_0^{-1}\int_{\Sigma}\dfrac{4\log(\beta_0) + 4\chi^+\overline{R}_N(z)}{\chi}d\bar\mu_{\Sigma}(z)\\
                   &   =  & C_0^{-1}\dfrac{4\log(\beta_0)}{\chi} + 4\dfrac{\chi^+}{\chi}\dfrac{N}{\bar\mu(\Sigma)}\\
                   &  <   & +\infty,
\end{eqnarray*}
by (\ref{Kac.lemma.N}).
\end{proof}

\begin{remark}
Let $f^R: \bigcup_iB_i \to B$ be a maximal induced Markov map for a rational map of the Riemann sphere. It is proved in \cite[Proposition 34]{doob.2012} 
that
$$
R(x) = \inf_{z \in \Sigma: \pi(z) = x}\overline{R}(z)
$$
This is a consequence of the conformal character of $f$. In particular, $R_i = \inf_{B_j \subset B_i}\overline{R}_j$ and therefore, those 
maximal induced Markov maps are natural or naturally extendible. See \cite[Lemma 2]{bruin.1995}. Nevertheless, we notice that our maximal 
induced Markov maps are not necessarily natural neither naturally extendible. So far if $f^{R_i}: B_i \to B$ is a branch of a 
maximal induced Markov map then there might be $0 < k < R_i$ and subsets $A \subset B_i$ such that $f^k: A \to B$ is a diffeomorphism, in other 
words, (\ref{R.bounds}) do not exclude the possibility of earlier returns. This seems to be a consequence of the non conformal character of 
$f$ and the non-uniformity of expansion. Notice that this may be the case even if $f$ is asymptotically conformal, that is, if all its Lyapunov 
exponents are equal.
\end{remark}

\begin{proof}[Proof of main technical lemma \ref{lema.tecnico.principal.1}]
Let us denote ${\Sigma}_m = \{ z \in \Sigma: \overline{R}(z) = m\}$. Then, we claim that there exists a finite collection of sub-rectangles 
$S^m_i \subset \Sigma$ such that 
$$
{\Sigma}_m = \bigcup_iS^m_i \quad\bar\mu-\text{modulo zero.}
$$
and $u$-rectangles $U^m_i$ such that $F^m: S_i^m \to U_i^m$ is a laminated diffeomorphism satisfying (\ref{expansion.uniforme.2}) and 
(\ref{forward.distorsion.limitada.2}) in the horizontal plaques $S^m_i(z)$.

The proof of the lemma follows since $\Sigma = \bigcup_{m > 0}\bigcup_iS^m_i$, $\bar\mu$-modulo zero. We get a family of laminated diffeomorphisms 
satisfying (\ref{expansion.uniforme.2}) and (\ref{forward.distorsion.limitada.2}) by relabeling the family $\{F^m : S^m_i \to U^m_i; m > 0, i > 0\}$.

By the property of regular returns, if $\overline{R}(z) = m$ then $\overline{R}(w) = m$ for every 
$w \in F^{-m}\Sigma_{F^m(z)}$.  In particular, 
$$
{\Sigma}_m = \bigcup_{z \in {\Sigma}_m}F^{-m}\Sigma_{F^m(z)}.
$$
However $\{\pi{F^{-m}\Sigma_{F^m(z)}}: z \in {\Sigma}_m\} = \{B^m_i\}$  is an at most countable family of open disjoint subsets. Indeed, if 
$\pi{F^{-m}\Sigma_{F^m(z')}} \subset \pi{F^{-m}\Sigma_{F^m(z)}}$ then $\pi{F^{-m}\Sigma_{F^m(z')}} = \pi{F^{-m}\Sigma_{F^m(z)}}$ since 
$f^m|\pi{F^{-m}\Sigma_{F^m(z)}}$ and $f^m|\pi{F^{-m}\Sigma_{F^m(z')}}$) are diffeomorphisms onto $B$. Otherwise, there would be a contradiction 
since $f^m|\pi{F^{-m}\Sigma_{F^m(z)}}$ is one-to-one.

Therefore, ${\Sigma}_m$ is a countable disjoint union of rectangles $S^m_i := \pi^{-1}(B^m_i) \cap \Sigma_m$, that is,
$$
{\Sigma}_m = \bigcup_{i: \bar\mu(S^m_i) > 0}S^m_i \quad\bar\mu-\text{modulo zero}.
$$ 

By the bounded distortion property there is $C(m) > 0$ such that $\Vol(B^m_i) \geq C(m)$ for every $i$, then there are at most finite $i$ for each $m>0$. 
Actually, there exists a constant $C > 0$ such that 
$$
e^{-C} \leq \Vol(B^m_i)Jf^m(x) \leq e^C
$$
Therefore, for every $i > 0$,
$$
\Vol(B^m_i) \geq C(n) := e^{-C}\sup_{x \in M}Jf(x)^{-m} > 0,
$$
so there are at most finitely many open sets $B^m_i$'s. 

We define $U^m_i := F^m(S^m_i)$, where $S^m_i$ has $\bar\mu(S^m_i) > 0$. As $S^m_i = \bigcup_{z \in S^m_i}F^{-m}\Sigma_{F^m(z)}$ then $U^m_i$ is an $u$-rectangle and 
$F^m: S^m_i \to U^m_i$ a laminated diffeomorphism. Estimatives (\ref{expansion.uniforme.2}) and (\ref{forward.distorsion.limitada.2}) follows 
from (\ref{expansion.uniforme.1}) and (\ref{forward.distorsion.limitada.1}) in the definition of rectangles with regular returns, since 
$F^{-m}\Sigma_{F^m(z)}$ is precisely the horizontal plaque $S^m_i(z)$ of height $z$ in $S^m_i$ and $\Sigma_{F^m(z)}$ the plaques of $U^m_i$ 
of level $F^m(z)$. To conclude, we simply relabel the family 
$$
\{F^m : S^m_i \to U^m_i: \ \bar\mu(S^m_i) > 0, \ m > 0, \ i > 0\}
$$
to get a family of hyperbolic branches $\{F^{\overline{R}_j}: S_j \to U_j\}$ as claim lemma \ref{lema.tecnico.principal.1} to exists.
\end{proof}

\begin{remark}
 Main technical lemma corrects a gap in the statement and proof of \cite[Lemma 4.3]{sanchez.salas.2003} which, essentially, amounts to say that
 $$
 \pi^{-1}(B^m_i) \cap \Sigma_m = \pi^{-1}(B^m_i) \cap \Sigma,
 $$
 that is, every point $z$ in the 'vertical' rectangle $\mathcal{U}_i^m = \pi^{-1}(B^m_i) \cap \Sigma$ returns for the first time to $\Sigma$ after $m$ 
 iterates. However, we can not exclude the possibility that some points $z \in \mathcal{U}_i^m$ return later to $\Sigma$. Indeed the image 
 of a vertical fiber $K_i = \mathcal{U}_i^m \cap \mathcal{F}(z)$ might not match the vertical fiber $K_{F^m(z)} = \Sigma \cap \mathcal{F}(F^m(z))$, even if 
 $F^m(z) \in \Sigma$, since $F^m(w)$ could fall into a gap of $K_{F^m(z)}$. All what we know for certain is that the set 
 $S^m_i := \pi^{-1}(B^m_i) \cap \Sigma_m$ is a rectangle with positive $\bar\mu$-measure, contained in the vertical strip $\mathcal{U}_i^m$.
 
 However, as we will show in next section, $\tau$ still has the structure of a horseshoe with infinitely many branches and 
 variable return times which projects onto an induced Markov map, as stated in \cite[Lemma 4.3]{sanchez.salas.2003}, even though points in the 
 branches of the horseshoe return for the first time at different moments. We explain this in next section.
\end{remark}

\section{First return map is a horseshoe with infinitely many branches}\label{sec:herraduras} 

\begin{definition}
We say that $\mathcal{S} \subset \Sigma$ is a $s$-rectangle if it is a rectangle over an open set $\pi(\mathcal{S}) \subset B$ and 
$K_z \cap \mathcal{S} = K_z$, where $K_z := \Sigma \cap \mathcal{F}(z)$ is the vertical fiber of $\Sigma$ passing by $z$, 
$\mathcal{F}(z) = \pi^{-1}(\pi(z))$.
\end{definition}

In other words, $\mathcal{S} \subset \Sigma$ is a \emph{a 'vertical strip' crossing $\Sigma$ from 'top to bottom'}.

We denote $\mathcal{S}(z)$ the plaque of 'height' $z \in \Sigma$ of a $s$-rectangle $\mathcal{S}$:
$\mathcal{S}(z) := \mathcal{S} \cap \Sigma_z$. Likewise for $u$-rectangles. 

A hyperbolic branch is a laminated map $\phi : \mathcal{S} \to \mathcal{U}$ which expands uniformly horizontal leaves of 
$\mathcal{S}$ and contracts uniformly the vertical fibers. More precisely,

\begin{definition}
Let $\mathcal{S}$ (resp. $\mathcal{U}$) be an $s$-rectangle (resp. $u$-rectangle) of $\Sigma$. A laminated $C^2$ map 
$\phi : \mathcal{S} \to \mathcal{U}$ is a \textbf{multivalued hyperbolic branch} with bounded distortion if there exists constants 
$C > 0$ and $\lambda > 1$ such that: 
\begin{enumerate}
\item for every $z \in \mathcal{S}$ there is an open subset $R(z) \subset \mathcal{S}(z)$ such that  
$\phi(R(z)) = \mathcal{U}({F^{n}(z)})$ and 
\begin{equation}\label{expansion.uniforme.3}
\|D\phi(w)v\| \geq C^{-1}\lambda\|v\|, \quad \forall \ v \in T_{w}\mathcal{S}(z), \ \forall \ w \in R(z);
\end{equation}
\item let $J\phi(z)$ the Jacobian of $\phi|\mathcal{S}(z)$ with respect to the volume element on $\Sigma_z$ given by the parametrization
$\Phi_z : B \to \Sigma_z$. Then,
\begin{equation}\label{forward.distorsion.limitada.3}
\dfrac{J\phi(w)}{J\phi(w')} \leq \exp(Cd(\phi(w),\phi(w')), \quad \forall \ w,w' \in R(z), \ \forall \ z \in \mathcal{S};
\end{equation}
\end{enumerate}
\item $\phi$ contracts exponentially the fibers $K_z \subset \mathcal{S}$.
\end{definition} 
 
\begin{remark}
$\phi : \mathcal{S} \to \mathcal{U}$ is \emph{'multivalued'} in the following sense: every horizontal plaque $\mathcal{S}(z)$ is decomposed 
into at most countably many open disjoint $R_i \subset \mathcal{S}(z)$ such that for each one of them is mapped onto a horizontal plaque 
$\Sigma_{i}$.
\end{remark}

The following result clarifies the structure of the first return transformation to $\Sigma$. 

\begin{proposition}\label{lema.herradura}\emph{(Horseshoe lemma)}
Let $\tau: \Sigma \itself$ be the first return map to a rectangle with regular returns $\Sigma$. Then, there exists a decomposition of $\Sigma$ into 
countably many disjoint $s$-rectangles $\mathcal{S}_i$ and $u$-rectangles $\mathcal{U}_i$ and multivalued hyperbolic branches 
$\phi_i : \mathcal{S}_i  \to \mathcal{U}_i$ such that
$$
\bar\mu\left(\Sigma - \bigcup_i\mathcal{S}_i\right) = \bar\mu\left(\Sigma - \bigcup_i\mathcal{U}_i\right) = 0
$$
and $\tau|\mathcal{S}_i = \phi_i$, that is, $\tau$ has the structure of a 'horseshoe with infinitely many branches'.
\end{proposition}

\begin{remark}
 Notice that the hyperbolic branches $\phi_i : \mathcal{S}_i \to \mathcal{U}_i$ do not expand strictly in the horizontal direction, due to the presence of a 
 constant $C^{-1}$ in (\ref{expansion.uniforme.3}). However, as we argued in the proof of theorem \ref{main.0}, we can choose a suitable iterate $\tau^N$ 
 endowed with a decomposition into countably many multivalued hyperbolic branches mapping 'vertical' rectangles $\mathcal{S}_i$ onto 'horizontal' rectangles 
 $\mathcal{U}_i$ in $\Sigma$ with good distortion which expands uniformly the horizontal plaques.
\end{remark}

\begin{proof}
By lemma \ref{lema.tecnico.principal.1}, $\tau$ is the union of a countable family of laminated diffeomorphisms $F^{\overline{R}_j}: S_j\mapto U_j$, 
satisfying (\ref{expansion.uniforme.3}) and (\ref{forward.distorsion.limitada.3}) with $\lambda = e^{-\chi/3}$. Let $f^R: \bigcup_iB_i \to B$ 
be a maximal induced Markov map and define 
$$
\mathcal{S}_i = \pi^{-1}(B_i) \cap \Sigma.
$$
By the main technical lemma each $\mathcal{S}_i$ has a decomposition $\mathcal{S}_i = \bigcup\{S_j: S_j \subset \mathcal{S}_i\}$, $\bar\mu$-modulo zero.
Let $\mathcal{U}_i = \bigcup\{U_j: \pi(S_j) \subset B_i\}$ and define
$$
\phi_i : \mathcal{S}_i \to \mathcal{U}_i \quad\text{as}\quad \phi_i|S_j = F^{\overline{R}_j}: S_j \to U_j.
$$
Then, $\mathcal{S}_i$ (resp. $\mathcal{U}_i$ is a $s$-rectangle (resp. $u$-rectangle) and clearly $\tau = \bigcup_i\phi_i$. 
By (\ref{expansion.uniforme.2}) and (\ref{forward.distorsion.limitada.2}), the maps $\phi_i : \mathcal{S}_i \to \mathcal{U}_i$ satisfy 
(\ref{expansion.uniforme.3}) and (\ref{forward.distorsion.limitada.3}), up-to parametrizations of the horizontal plaques.

To prove that each $\phi_i$ contracts exponentially the vertical fibers in $\Sigma$ we notice that, by a simple calculation, for every $n > 0$, 
$d(F^n(w),F^n(w')) \leq 1/2^nd(w,w')$, for every $w,w' \in \mathcal{F}(z)$. Now, let $w,w' \in K_z$ be two points on the same vertical fiber $K_z = \mathcal{F}(z) \cap \Sigma$. Then,
\begin{eqnarray*}
 d(\tau^n(w),\tau^n(w')) &   =  & d\left(F^{\overline{R}_n(w)}(w), F^{\overline{R}_n(w')}(w')\right)\\
                         & \leq & 2^{-\overline{R}_n(w)}d\left(w, F^{\overline{R}_n(w')-\overline{R}_n(w)}(w')\right)\\
                         & \leq & 2^{-n}\diam(Z)
\end{eqnarray*}
since $\overline{R}_n \geq n$. That is, $d(\tau^n(w),\tau^n(w')) \to 0$ exponentially, for every pair of points $w,w' \in K_z$ at the same vertical fiber 
of $\Sigma$. As $\tau|\mathcal{S}_i = \phi_i$ this shows that $\phi_i$ contracts vertical fibers exponentially.
\end{proof}

\end{document}